\documentclass[11pt]{amsart}

\newtheorem{theorem}{Theorem}[section]

\newtheorem{conjecture}[theorem]{Conjecture}
\newtheorem{corollary}[theorem]{Corollary}

\newtheorem{definition}[theorem]{Definition}
\newtheorem{lemma}[theorem]{Lemma}

\newtheorem{proposition}[theorem]{Proposition}
\newtheorem{remark}[theorem]{Remark}

\usepackage[colorlinks, bookmarks=true]{hyperref}
\usepackage{color,graphicx,shortvrb}
\usepackage[active]{srcltx} 

\def\J#1#2#3{ \left\{ #1,#2,#3 \right\} }

\def\11{\textbf{$1$}}

\usepackage{enumerate}
\usepackage{amsmath}
\usepackage{amssymb}

\begin{document}

\title[A Kadec-Pelczy{\'n}ski dichotomy-type theorem]{A Kadec-Pelczy{\'n}ski dichotomy-type theorem for preduals of JBW$^*$-algebras}

\author[Fern\'{a}ndez-Polo]{Francisco J. Fern\'{a}ndez-Polo}
\email{pacopolo@ugr.es}
\address{Departamento de An{\'a}lisis Matem{\'a}tico, Facultad de
Ciencias, Universidad de Granada, 18071 Granada, Spain.}

\author[Peralta]{Antonio M. Peralta}
\email{aperalta@ugr.es}
\address{Departamento de An{\'a}lisis Matem{\'a}tico, Facultad de
Ciencias, Universidad de Granada, 18071 Granada, Spain.}

\author[Ram\'{i}rez]{Mar{\'\i}a Isabel Ram{\'\i}rez}
\address{Departamento de Algebra y An\'alisis Matem\'atico, Universidad de
Almer\'ia, 04120 Almer\'ia, Spain} \email{mramirez@ual.es}

\thanks{First and second author partially supported by the Spanish Ministry of Economy and Competitiveness,
D.G.I. project no. MTM2011-23843, and Junta de Andaluc\'{\i}a grants FQM0199 and
FQM3737. Third author supported partially supported by the Spanish Ministry of Economy and Competitiveness project no. MTM2010-17687.}



\begin{abstract} We prove a Kadec-Pelczy{\'n}ski dichotomy-type theorem for bounded sequences in the predual of a JBW$^*$-algebra, showing that for each bounded sequence $(\phi_n)$ in the predual of a JBW$^*$-algebra $M$, there exist a subsequence $(\phi_{\tau(n)})$,
and a sequence of mutually orthogonal projections $(p_n)$ in $M$ such that: \begin{enumerate}[$(a)$]
\item the set $\left\{\phi_{\tau(n)} - \phi_{\tau(n)} P_{2} (p_n): n\in \mathbb{N}\right\}$ is relatively weakly compact,
\item $\phi_{\tau(n)}=\xi_n+\psi_n$, with $\xi_n := \phi_{\tau(n)} - \phi_{\tau(n)} P_{2} (p_n)$, and $\psi_n := \phi_{\tau(n)} P_{2} (p_n),$ {\rm(}$\xi_n Q(p_n)= 0$ and $\psi_n Q(p_n)^2 = \psi_n${\rm)}, for every $n$.
\end{enumerate}
\end{abstract}


\maketitle
 \thispagestyle{empty}

\section{Introduction}

A celebrated and useful result of M.I. Kadec and A. Pelczy{\'n}ski \cite{KadPelcz} (cf. the splitting technique in \cite[page 97]{Die}) states that every bounded sequence $(f_n)$ in $L_p [0,1]$  $(1\leq p<\infty)$ admits a subsequence $(f_{n_{k}})$ which can be written as $g_k +h_k$, where the $h_k$'s have pairwise disjoint supports and the set $\{ g_k : k\in \mathbb{N}\}$ is an equi-integrable or relatively weakly compact subset in $L_p[0,1]$. This result, nowadays called ``subsequence splitting property'' or ``Kadec-Pelczy{\'n}ski dichotomy theorem'', doesn't hold for Banach function spaces in general (examples include $c_0$, reflexive, p-convex Banach lattices \cite{FigGhouJohn}). However, the number of subsequent contributions motivated by the above result is enormous; trying to refer to all the very large number of publications and authors involved is almost impossible.\smallskip

For our purposes, we focus on a generalisation of the subsequence splitting property for preduals of (non-necessarily commutative) von Neumann algebras. When the space $L_1 [0,1]$ is regarded as the predual of the abelian von Neumann algebra $L_{\infty}[0,1]$, we are naturally led to the following question: Does the Kadec-Pelczy{\'n}ski dichotomy theorem hold for preduals of (non- necessarily commutative) von Neumann algebras? An affirmative answer to this question was given by N. Randrianantoanina in \cite{Randri} and by Y. Raynaud and Q. Xu in \cite{RaynaudXu}.\smallskip

A Kadec-Pelczy{\'n}ski dichotomy-type problem also makes sense in the wider non-associative context of JB$^*$-algebras. Let us recall that an algebra $J$ with a commutative product (written $a\circ b$) is
called a \emph{Jordan algebra} if the identity $a\circ (b \circ a^2) = (a\circ b)\circ a^2,$ holds for every $a,b$ in $J$. A JB$^*$-algebra is a complex Banach space $J$ which is a complex Jordan algebra, with product $\circ$,
equipped with an involution $^*$ satisfying $$\|a\circ b\|\leq \| a\| \ \|b\|, \ \|a^*\| = \|a\| \hbox{ and }\|\J aaa\| =\|a\|^3,$$ for every $a,b\in J$, where $\J aaa = 2 (a\circ a^*) \circ a - a^2 \circ a^*$. Given a JB$^*$-algebra $J$, the set $J_{sa}$ of its self-adjoint elements is a JB-algebra in the sense of \cite{HancheStor}. Conversely, every JB-algebra coincides with the self adjoint part of a JB$^*$-algebra (cf. \cite{Wri77}). A JBW$^*$-algebra is a JB$^*$-algebra $M$ which is also a dual Banach space. The bidual of every JB$^*$-algebra is a JBW$^*$-algebra. It is known that every JBW$^*$-algebra $M$ has a unique (isometric) predual (denoted by $M_*$) and the product of $M$ is separately weak$^*$-continuous (see \cite[\S 4]{HancheStor}). Elements in $M_*$ are called normal functionals.\smallskip

Every C$^*$-algebra (respectively, every von Neumann algebra) is a JB$^*$-algebra (respectively, a JBW$^*$-algebra) with respect to its (usual) Jordan product $a\circ b := \frac 12 (a b +ba)$and its natural involution. A JC$^*$-algebra is any norm-closed Jordan $^*$-subalgebra of a C$^*$-algebra, while a JW$^*$-algebra is a weak$^*$-closed Jordan $^*$-subalgebra of a von Neumann algebra.\smallskip

We devote this paper to establish a Kadec-Pelczy{\'n}ski dichotomy-type theorem for bounded sequences in the predual of a JBW$^*$-algebra (see Theorem \ref{t KP dichotomy for general sequences 1} and Corollary \ref{c KP dichotomy orthogonal for general sequences}). We concretely prove that for each bounded sequence $(\phi_n)$ in the predual of a JBW$^*$-algebra $M$, there exist a subsequence $(\phi_{\tau(n)})$,
and a sequence of mutually orthogonal projections $(p_n)$ in $M$ such that: \begin{enumerate}[$(a)$]
\item the set $\left\{\phi_{\tau(n)} - \phi_{\tau(n)} P_{2} (p_n): n\in \mathbb{N}\right\}$ is relatively weakly compact,
\item $\phi_{\tau(n)}=\xi_n+\psi_n$, with $\xi_n := \phi_{\tau(n)} - \phi_{\tau(n)} P_{2} (p_n)$, and $\psi_n := \phi_{\tau(n)} P_{2} (p_n),$ {\rm(}$\xi_n Q(p_n)= 0$ and $\psi_n Q(p_n)^2 = \psi_n${\rm)}, for every $n$.
\end{enumerate}

In Section 2  we revisit some characterisations of relatively weakly compact subsets in the predual $M_*$ of a JBW$^*$-algebra $M$, connecting this definition with the notion of uniformly integrable subsets in $M_*.$

\subsection{Background and notation}

For each element $a$ in a JB$^*$-algebra $J,$ the symbol $M_a$ will denote the multiplication operator on $J$ defined by $M_a (c) := a \circ c $ ($c\in J$). Given another element $c$ in $J$, the conjugate-linear mapping $Q(a,b):J \to J$ is given by $Q(a,b) (c) := \J acb,$ where the triple product of three elements $a,b,c$ in $J$ is defined by $\J acb= (a\circ c^*) \circ b + (b\circ c^*)\circ a - (a\circ b)\circ c^*.$ We write $Q(a)$ for $Q(a,a)$. Clearly, $Q(a,b) = Q(b,a),$ for every $a,b\in J.$ It is known that the identity
\begin{equation}\label{basic equation}
Q(a) Q(b) Q(a) = Q(Q(a)b), \end{equation}
holds for every $a,b\in J.$ The mapping $L(a,b) : J \to J$ is defined by $L(a,b) (c) := \J abc$ ($c\in J$).\smallskip

Given a Banach space $X$, the symbol $B(X)$ will stand for its closed unit ball.\smallskip

By a projection $p$ in a JBW$^*$-algebra $M$ we mean a self-adjoint idempotent element in $M$ (i.e. $p=p^*= p\circ p$). The set of all projections in $M$ will be denoted by Proj$(M)$. Each projection $p$ in $M$ induces a decomposition of $M$ $$M= M_{2} (p) \oplus M_{1} (p) \oplus M_0 (p),$$ where for $j=0,1,2,$ $M_j (p)$
is the $\frac{j}{2}$-eigenspace of the mapping $L(p,p)$. This is called the \emph{Peirce decomposition} of $M$ with respect to $p$ (\cite[\S 2.6]{HancheStor}). The following multiplication rules are satisfied:
$$\{M_{i_{}}(p),M_{j_{}}(p),M_{k_{}}(p)\}\subseteq
M_{i-j+k}(p),$$ where $i,j,k=0,1,2$ and $M_{l_{}}(p)=0$ for $l\neq 0,1,2,$
$$\{M_0(p),M_2(p),M\}=\{M_2(p),M_0(p),M\}=0.$$

We recall that a normal functional $\varphi$ in the predual of a JBW$^*$-algebra $M$ is called positive when $\varphi (a) \geq 0$ for every $a\geq 0$ in $M$. It is know that a normal functional $\varphi$ is positive if and only if $\varphi (1) = \|\varphi\|.$ A normal state is a norm-one positive normal functional.\smallskip

The strong$^*$-topology (denoted by $S^*(M,M_*$) of a JBW$^*$-algebra $M$ is the topology on $M$ generated by
the family of seminorms of the form $$x\mapsto \|x\|_{\varphi}:=\sqrt{\varphi (x \circ
x^*)},$$ where $\varphi$ is any positive normal functional in $M_{*}$
\cite[\S 4]{PeRo}. Consequently, when a von Neumann algebra $W$ is regarded as a complex JBW$^*$-algebra, $S^{*} (W,W_{*})$ coincides with the familiar strong*-topology of $W$ (compare \cite[Definition 1.8.7]{Sak}). It is known that the strong*-topology of $M$ is compatible with the duality $(M,M_{*})$. Moreover, the triple product of $M$ is
jointly $S^{*} (M,M_{*})$-continuous on bounded subsets of $M$ (cf. \cite[\S 4]{PeRo}).

\section{Preliminaries: Weak Compactness and Uniform Integrability}

In this section we review some basic characterisations of relatively weakly compact subsets in the predual, $M_*$, of a JBW$^*$-algebra $M$,
relating the latter notion with the concept of (countably) uniformly integrable subsets in $M$.
The following definition is motivated by the corresponding notion introduced by N. Randrianantoanina \cite{Randri}
(see also Y. Raynaud and Q. Xu \cite{RaynaudXu}) in the setting of von Neumann algebras.

\begin{definition}\label{def count uniformly int} A bounded subset, $K,$ in the predual, $M_*$, of a JBW$^*$-algebra $M$ is said to be (countably) uniformly integrable if for every strong$^*$-null, bounded sequence
$(a_n)$ in $M$ we have $$\lim_{n\to \infty} \sup \left\{ \| \varphi \ Q({a_n})\| : \varphi\in K\right\}=0.$$
\end{definition}


When $M$ is a von Neumann algebra, a bounded subset $K\subseteq M_*$ is (countably) uniformly integrable if, and only if, $K$ is relatively weakly compact (cf. \cite[Page 140]{Randri} and \cite[Lemma 4.3 and Proposition 4.4 or Corollary 4.9]{RaynaudXu}). In the setting of preduals of JBW$^*$-algebras the above equivalence has not been established. We devote this section to explore the connections appearing in this case.\smallskip

M. Takesaki \cite{Tak0}, C.A. Akemann \cite{Ak}, C.A. Akemann, P.G. Dodds, J.L.B. Gamlen \cite{AkDoGam} and K. Sait{\^o} \cite{Sa} provided several useful characterisations of relatively weakly compact subsets of
preduals of von Neumann algebras, characterisations which are exploited in \cite{Randri} and \cite{RaynaudXu}.
In the setting of preduals of JBW$^*$-algebras, we shall require the following characterisation which was established in \cite{Pe}.

\begin{theorem}\label{th weak compact char alg}\cite[Theorem 1.5]{Pe}
Let $K$ be a bounded subset in the predual, $M_*$, of a JBW$^*$-algebra $M$.
The following assertions are equivalent:\begin{enumerate}[$(i)$]
\item $K$ is relatively weakly compact; \item The
restriction $K|_{C}$ of $K$ to each maximal associative subalgebra
$C$ of $M$ is relatively $\sigma (C_{*},C)$-compact; \item
There exists a (positive) normal state $\psi\in M_{*}$ having the following
property: given $\varepsilon >0$ there exists $\delta >
0$ such that for every $x\in W$ with $\|x\|\leq 1$ and
$\|x\|_{\psi} < \delta,$ then $| \phi (x)| < \varepsilon$ for each
$\phi\in K;$ \item For any monotone decreasing sequence of
projections $(p_n)$ in $M$ with $(p_n) \to 0$ in the
weak*-topology, we have $\lim_{n\to +\infty} \phi (p_n) =0,$
uniformly for $\phi\in K$. $\hfill \Box$
\end{enumerate}
\end{theorem}

It is known that, for each element $a$ in a von Neumann algebra $A$, the mapping $Q({a,a^*}) (x) :=\frac12 ( a x a^* + a^* x a)$
is a positive operator. Thus, given $0\leq z$ in $A$, we have $0\leq Q({a,a^*}) (z) \leq \|z\| Q({a,a^*}) (1)$.\smallskip

We shall need the following variant of \cite[Proposition 1.8.12]{Sak}.
We firstly recall that the product of a von Neumann algebra is jointly strong$^*$-continuous on bounded sets
(cf. \cite[Proposition 1.8.12]{Sak}).

\begin{lemma}\label{l 1 von Neumann alg unif conv of Ua}
Let $(a_{n})$ be a bounded, strong$^*$-null sequence in a von Neumann algebra $A$.
Then $$\hbox{\rm strong$^*$-}\lim_{n} Q({a_{n}}) (x) = 0,$$ uniformly on $B(A)$.
\end{lemma}

\begin{proof}
We may assume, without loss of generality, that $a_{n}\in B(A)$ for every $n$.
Fix an arbitrary element $x$ in $B(A)$ and a normal state $\varphi\in A_*$. The inequalities
$$\left\| Q({a_{n}}) (x) \right\|_{\varphi}^2 = \frac12 \varphi \left( a_{n} x a_{n} a_{n}^* x^* a_{n}^* +
a_{n}^* x^* a_{n}^* a_{n} x a_{n}\right)$$
$$\leq  \frac12 \left( \| x a_{n} a_{n}^* x^*\| \varphi (a_{n} a_{n}^*)  +
 \|x^* a_{n}^* a_{n} x\| \varphi( a_{n}^* a_{n})\right) $$
$$\leq \frac 12  \varphi(a_{n} a_{n}^*+ a_{n}^* a_{n})= \|a_{n}\|_{\varphi}^2,$$ imply the desired statement.
\end{proof}

The following fact, proved by L. Bunce in \cite{Bun01}, will be used throughout the paper: Let $J$ be a JBW$^*$-subalgebra (i.e. a weak$^*$-closed
Jordan $^*$-subalgebra) of a JBW$^*$-algebra $M$, then the strong$^*$-topology of $J$ coincides with the restriction to $J$ of the strong$^*$-topology of $M$,
that is, $S^*(J,J_*) = S^* (M,M_*)|_{J}$. We shall actually require a weaker version of this result asserting that $S^*(J,J_*)$  and $S^* (M,M_*)|_{J}$ coincide on bounded subsets of $J$ (compare \cite[\S 4]{PeRo}).\smallskip

Since every JW$^*$-algebra is a weak$^*$-closed Jordan $^*$-subalgebra of a von Neumann algebra,
the following result is a direct consequence of Lemma \ref{l 1 von Neumann alg unif conv of Ua} and the above fact.

\begin{corollary}\label{c 1 JW alg unif conv of Ua}
Let $(a_{n})$ be a bounded, strong$^*$-null sequence in a JW$^*$-algebra $M$.
Then $$\hbox{\rm strong$^*$-}\lim_{n} Q({a_{n}}) (x) = 0,$$ uniformly on $B(M)$.$\hfill\Box$
\end{corollary}

The proof of the next lemma is left to the reader.

\begin{lemma}\label{l direct sums of finite-dimensional} Let $(M_j)_{j\in \Lambda}$ be a family of JBW$^*$-algebras
satisfying the following property: for each bounded, strong$^*$-null sequence $(a_{n})$ in $M_j$ we have $$\hbox{\rm strong$^*$-}\lim_{n} Q({a_{n}}) (x) = 0,$$ uniformly on $B(M_j)$. Then the JBW$^*$-algebra $M= \bigoplus^{\ell_{\infty}}_{j\in \Lambda} M_j$ satisfies the same property, that is,
for each bounded, strong$^*$-null sequence $(a_{n})$ in $M$ we have $$\hbox{\rm strong$^*$-}\lim_{n} Q({a_{n}}) (x) = 0,$$ uniformly on $B(M)$.$\hfill\Box$
\end{lemma}



It is well known that in a finite dimensional JBW$^*$-algebra the strong$^*$-topology, the weak, the weak$^*$ and the norm topologies coincide.\smallskip

By \cite[Theorem 7.2.7]{HancheStor} every JBW$^*$ algebra $M$ can be (uniquely)
decomposed as a direct $\ell_{\infty}$-sum $M = M_1 \bigoplus^{\ell_{\infty}} M_2$, where $M_1$ is a JW$^*$-algebra
and $M_2$ is a purely exceptional JBW$^*$-algebra. It is further known that $M_2$ embeds as a JBW$^*$-subalgebra of an $\ell_{\infty}$-sum of finite-dimensional exceptional JBW$^*$-algebras (compare \cite[Lemma 7.2.2 and Theorem 7.2.7]{HancheStor}). Combining these structure results and the preceding comments with Corollary \ref{c 1 JW alg unif conv of Ua} and Lemma \ref{l direct sums of finite-dimensional} we get:

\begin{corollary}\label{c 1 JBW alg unif conv of Ua}
Let $(a_{n})$ be a bounded, strong$^*$-null sequence in a JBW$^*$-algebra $M$.
Then $$\hbox{\rm strong$^*$-}\lim_{n} Q({a_{n}}) (x) = 0,$$ uniformly on $B(M)$.$\hfill\Box$
\end{corollary}

We can prove now the equivalence between relatively weakly compact subsets and (countably) uniformly integrable subsets in the predual of a JBW$^*$-algebra.

\begin{proposition}
\label{p equiv rwc and uniform integrable} The following statements are equivalent for any bounded subset $K$ in the predual of a JBW$^*$-algebra $M$.
\begin{enumerate}[{\rm $(a)$}] 
\item $K$ is relatively weakly compact;
\item $K$ is (countably) uniformly integrable;
\item For each strong$^*$-null sequence $(p_{n}) \subset \hbox{Proj}(M)$, we have $$\exists \lim_{n\to \infty} \sup_{\varphi\in K} \| \varphi Q({p_{n}}) \| =0.$$
\item For each strong$^*$-null decreasing sequence $(p_{n}) \subset \hbox{Proj}(M)$, we have $$\exists \lim_{n\to \infty} \sup_{\varphi\in K} \| \varphi Q({p_{n}}) \| =0.$$
\end{enumerate}
\end{proposition}

\begin{proof}
We may assume that $K\subset B(M_*)$.
Since for each projection $p$ in $M$ and each $\varphi\in K$, $|\varphi(p)|= |\varphi(Q(p) p)| \leq \|\varphi Q(p)\|$, the implications
{\rm $(b)\Rightarrow (c)$} and {\rm $(c)\Rightarrow (d)$} are clear, while {\rm $(d)\Leftrightarrow (a)$} follows from Theorem \ref{th weak compact char alg}. \smallskip

{\rm $(a)\Rightarrow (b)$}.  By Theorem \ref{th weak compact char alg}~$(i)\Leftrightarrow (iii)$, there exists a normal state $\phi\in M_*$ satisfying that for each $\varepsilon>0$, there exists $\delta>0$ such that for every $x\in B(M)$ with $\|x\|_{\phi} <\delta$, $|\varphi (x) | < \varepsilon$, for every $\varphi \in K$. Let $(a_{n})$ be a strong$^*$-null sequence in $B(M)$. Corollary \ref{c 1 JBW alg unif conv of Ua} asserts that there exists $n_0$ such that for each $n\geq n_0$, $\| Q(a_{n}) (x)\|_{\phi} < \delta$, for every $x\in B(M)$. Therefore, for every $n\geq n_0$, $|\varphi (Q({a_{n}}) (x)) | < \varepsilon$, for all $\varphi \in K$, $x\in B(M)$, which implies, $\|\varphi Q({a_{n}}) \| \leq \varepsilon$, for every $\varphi \in K$.\end{proof}

\begin{lemma}\label{l 2.3} Let $K$ be a bounded subset in the predual of a JBW$^*$-algebra $M$.
Let $(p_{n})$ be a strong$^*$-null sequence in Proj$(M)$. Suppose that, for each $n_0$,
$P_1 (p_{n_0})^* (K)$ and $P_0 (p_{n_0})^* (K)$ are relatively weakly compact. Then $K$ is relatively weakly compact if, and only if,
$\displaystyle \exists \lim_{n} \sup_{\varphi\in K} \| \varphi Q({p_{n}})\| =0.$ 
\end{lemma}

\begin{proof} The necessary condition being obvious, let us suppose that $P_1 (p_{n_0})^* (K)$ and $P_0 (p_{n_0})^* (K)$ are relatively weakly compact and $\displaystyle\exists\lim_{n} \sup_{\varphi\in K} \| \varphi Q({p_{n}})\| =0.$ Let $(q_m)$ be a strong$^*$-null sequence of projections in $M$. By assumptions, for each $\varepsilon>0$, there exists $n_0$ such that $$| \varphi P_{2}({p_{n}}) (q_{m})| \leq \|\varphi Q({p_{n}})\| < \frac{\varepsilon}{3}, \hbox{for every } n\geq n_0, m\in \mathbb{N} \hbox{ and } \varphi\in K.$$ Find $m_0$ satisfying $$\left| \varphi P_0 (p_{n_0}) (q_{m}) \right| < \frac{\varepsilon}{3}, \hbox{ and } \left| \varphi P_1 (p_{n_0}) (q_{m}) \right| < \frac{\varepsilon}{3},$$ for every $m\geq m_0$, $\varphi\in K$. Finally $$|\varphi (q_{m})| \leq \sum_{i=0}^{2} \left|\varphi P_i (p_{n_0}) (q_{m}) \right| <\varepsilon,$$ for every $m \geq m_0,$ $\varphi\in K$.
\end{proof}

We state now an appropriated Jordan version of \cite[Lemma 2.4]{Randri}, the proof is left for the reader.

\begin{lemma}\label{l 2.4} Let $M$ be a JBW$^*$-algebra. Let $(p_n)$ be a strong$^*$-null decreasing sequence of projections in $M$,
$(\varphi_k)$ a bounded sequence in $M_*$ such that
$\displaystyle \lim_{n\to \infty} \sup_{k} \| \varphi_{k} Q({p_{n}})\| =\theta>0.$ The following statements hold: \begin{enumerate}[$(a)$]
\item For each natural $k_0$,
$\displaystyle \lim_{n\to \infty} \sup_{k\geq k_0} \| \varphi_{k} Q({p_{n}})\| =\theta>0$;
\item There exist subsequences $(p_{n_m})$, $(\varphi_{k_m})$ satisfying
$$\exists \lim_{m\to \infty} \| \varphi_{k_{m}} Q({p_{n_{m}}})\| =\theta.$$ $\hfill\Box$
\end{enumerate}
\end{lemma}

\section{A dichotomy-type theorem for bounded sequences of positive normal functionals}

This section is devoted to establish a Kadec-Pelczy{\'n}ski dichotomy-type theorem for bounded sequences in the positive part of the predual of a JBW$^*$-algebra $M$. The main result (cf. Theorem \ref{c KP dichotomy for positive sequences}) shows that for every bounded sequence $(\varphi_n)$ in $M_*^+$, there exist a subsequence $(\varphi_{n_k})$, bounded sequences $(\phi_k)$ and $(\psi_k)$ in $M_*$ and a decreasing strong$^*$-null sequence of projections $(q_n)$ in $M$ such that $\varphi_{n_k}=\phi_k+\psi_k$, $\phi_k Q(q_k)= 0$, $\psi_k P_2(q_k) = \psi_k$ and the set $\{\phi_k: k\in\mathbb{N}\}$ is relatively weakly compact.\smallskip

To achieve the result, we firstly establish a series of technical lemmas.\smallskip

We recall that elements $a$ and $b$ in a JB-algebra $\mathcal{A}$ are said to \emph{operator commute}
in $\mathcal{A}$ if the multiplication operators $M_a$ and $M_b$ commute.
It is known that $a$ and $b$ generate a JB-subalgebra that can be identified as a norm-closed Jordan subalgebra of
the self adjoint part of some $B(H)$, \cite{Wri77}, and, in this identification, $a$ and $b$
commute in the usual sense whenever they operator commute in $\mathcal{A}$
(compare Proposition 1 in \cite{Top}).  It also follows from the just quoted reference
that $a$ and $b$ operator commute if, and only if, $a^2 \circ b =\J aba$
(i.e., $a^2 \circ b = 2 (a\circ b)\circ a - a^2 \circ b$).\smallskip

\begin{lemma}
\label{l 3.3} Let $M$ be a JBW$^*$-algebra. Let  $(a_n)$ be a strong$^*$-null, decreasing sequence in $M$ with
$0\leq a_n\leq 1$, for every natural $n$.
Then for every $\varepsilon >0$ there exists a sequence of projections $(q_n)$ in $M$ satisfying
:\begin{enumerate}[$(a)$]
\item The sequence $(q_n)$ is strong$^*$-null and $q_n\leq q_1\leq r (a_1)$ for every $n\geq 1$;
\item For each $n_0$, $\displaystyle \left\| Q(a_n) -  Q(Q(q_{n_0}) (a_n))\right\|< \varepsilon,$  for every $n\geq n_0$.
\end{enumerate}
\end{lemma}

\begin{proof}
Fix $0< \delta <\min\{ \left( \frac{\varepsilon}{20}\right)^2, 1\}$. Define, by functional calculus,
$q_1:= \chi_{(\delta,1)} (a_1)$ and for each $n\geq 2$, $q_{n}:= \chi_{(\delta,1)} (\J {q_1}{a_n}{q_1}).$
Clearly, $$0\leq Q(q_1) (a_n)\leq Q(q_1) (1) = q_1,$$ therefore, $q_n\leq q_1$ for every natural $n$.
Since $0\leq \delta q_n \leq \J {q_1}{a_n}{q_1},$ and the latter defines a strong$^*$-null sequence,
we deduce that $(q_n) \to 0$ in the strong$^*$-topology.\smallskip

In order to simplify notation, let us denote $b_n := \J {q_1}{a_n}{q_1}$. Then $(b_n)$ is a strong$^*$-null,
decreasing sequence of positive elements in $B(M)$. Fix $n\geq n_0$.
If we write $b_n = Q(q_{n_0}) (b_n) + 2 Q(q_{n_0},1-q_{n_0}) (b_n) + Q(1-q_{n_0}) (b_n),$ we have
$$b_n - Q(q_{n_0}) (b_n) = 2 Q(q_{n_0},1-q_{n_0}) (b_n) + Q(1-q_{n_0}) (b_n).$$
Now, since $ Q(1-q_{n_0})$ is a positive operator, $$0\leq Q(1-q_{n_0}) (b_n) \leq Q(1-q_{n_0}) (b_{n_0}) \leq \delta 1.$$ For the other summand, we observe that, since $q_{n_0}$ and $1-q_{n_0}$ operator commute, $2 Q(q_{n_0},1-q_{n_0}) (b_n) = 4 ((1-q_{n_0})\circ (b_n))\circ q_{n_0}$. Applying \cite[Lemma 3.5.2 $(ii)$]{HancheStor}, we get $$2^2 \|Q(q_{n_0},1-q_{n_0}) (b_n)\|^{2} \leq 4^2 \|(1-q_{n_0})\circ (b_n)\|^2 $$ $$\leq 4^2 \|b_n\| \ \|\J {1-q_{n_0}}{b_n}{1-q_{n_0}}\|$$ $$= 4^2 \|b_n\| \ \|Q( {1-q_{n_0}}){b_{n_0}}\| \leq 4^2 \delta.$$ Therefore $$\|Q(q_1) (a_n) - Q(q_{n_0}) (a_n)\|=  \|b_n - Q(q_{n_0}) (b_n)\|\leq 5 \sqrt{\delta}.$$

By a similar argument we can check that $$\| a_n -b_n\| = \| a_n - Q(q_1) (a_n) \| \leq 2 \| Q(q_1, 1- q_1) (a_n) \| + \| Q(1-q_1) (a_n)\|,$$
$$0\leq Q(1-q_1) (a_n) \leq Q(1-q_1) (a_1) \leq \delta 1,$$ and
$$\| 2 Q(q_1, 1- q_1) (a_n) \|^2 = 2^2 \| 2 (a_n\circ (1-q_1))\circ q_1\|^2 \leq 4^2 \| a_n\circ (1-q_1)\|^2$$
$$\hbox{(by \cite[Lemma 3.5.2 $(ii)$]{HancheStor})} \leq 4^2 \| a_n\| \ \|Q(1-q_1) (a_n)\| \leq 4^2 \delta.$$ Combining the above facts we get
$\| a_n -b_n\|\leq 5 \sqrt{\delta},$ and $$\|a_n - Q(q_{n_0}) (a_n) \|\leq 10 \sqrt{\delta}.$$

Finally, the inequality $$\| Q(a_n) -  Q(Q(q_{n_0}) (a_n))\| \leq \|Q(a_n) -  Q(a_n,Q(q_{n_0}) (a_n))\| $$
$$+ \|Q(a_n,Q(q_{n_0}) (a_n)) -  Q(Q(q_{n_0}) (a_n))\| \leq 20 \sqrt{\delta},$$ concludes the proof.
\end{proof}

Let $(a_n)$ be a sequence of symmetric elements in a JBW$^*$-algebra, $M,$ such that $(a_n^2)$ is decreasing and $\|a_n\| \leq 1$ for every $n$. The mapping $Q(a_n)$ is positive (cf. \cite[Proposition 3.3.6]{HancheStor}), so for each positive (normal) functional $\varphi$ on $M$, $$\|\varphi Q(a_n)\|= \varphi Q(a_n) (1) = \varphi (a_n^2)\geq \varphi (a_{n+1}^2) = \|\varphi Q(a_{n+1})\|.$$ Thus, given a bounded subset $K$ in the positive part of $M_*$, the sequence $\Big(\sup_{\varphi\in K} \Big\{ \|\varphi Q(a_n)\| \Big\}\Big)_{n}$ is monotone decreasing and bounded, and hence convergent. If we only assume that  $(a_n)$ is decreasing we cannot assure that the above sequence is convergent, we shall consider, in this case, its lower limit $\displaystyle\liminf_{n} \sup_{\varphi\in K} \Big\{ \|\varphi Q(a_n)\| \Big\}$. It follows from the above arguments that the sequence $\Big(\sup_{\varphi\in K} \Big\{ \|\varphi Q(p_n)\| \Big\}\Big)_n$ converges, whenever $(p_n)$ is a decreasing sequence of projections in $M$.\smallskip

The following result is a direct consequence of Lemma \ref{l 3.3} above.

\begin{corollary}
\label{c l 3.3} Let $(\varphi_k)$ be a sequence of positive functionals
in the closed unit ball of the predual, $M_*$, of a JBW$^*$-algebra $M$.
Suppose there exist $\gamma>0$ and a strong$^*$-null, decreasing sequence $(a_n)$
of positive elements in the closed unit ball of $M$ satisfying that
$\displaystyle{\liminf_{n\to \infty}} \| \varphi_{n} Q({a_n}) \| \geq\gamma$
{\rm(}respectively, $\displaystyle{\liminf_{n\to \infty}} \sup_{k} \| \varphi_{k} Q({a_n}) \| \geq\gamma${\rm)}.
Then for every $\varepsilon >0$ there exists a sequence of projections $(q_n)$ in $M$ satisfying:\begin{enumerate}[$(a)$]
\item The sequence $(q_n)$ is strong$^*$-null and $q_n\leq q_1\leq r (a_1)$ for every $n\geq 1$;
\item For each natural $n_0$, $\displaystyle
{\liminf_{n\to \infty}} \| \varphi_{n} Q({Q({q_{n_0}})(a_n)}) \| \geq \gamma - \varepsilon$
{\rm(}respectively, $\displaystyle
{\liminf_{n\to \infty}} \sup_{k} \| \varphi_{k} Q({Q({q_{n_0}})(a_n)}) \| \geq \gamma - \varepsilon${\rm)}.$\hfill\Box$
\end{enumerate}
\end{corollary}

We recall that a positive functional $\psi$ on a JB$^*$-algebra $J$
is said to be faithful if $\psi (x) > 0$ for every
positive element $x\in J\backslash \{0\}$. Suppose that a
JBW$^*$-algebra $M$ admits a faithful normal state $\psi$. Then it is known that the
strong*-topology in the closed unit ball of $M$ is metrized by the
distance $$d_{\psi} (a,b) := \left(\psi( (a-b)\circ (a-b)^*
)\right)^{\frac{1}{2}} = \| a-b\|_{\psi}$$
(compare \cite[page 200]{Io}).\smallskip

Given a positive normal functional $\phi$ on a JBW$^*$-algebra $M$
then $\phi$ has a unique \emph{support or carrier projection} $p\in M$, such that $\phi = \phi P_2(p)$,
and $\phi$ is faithful on $Q(p) (M)=M_2 (p)$ (see \cite[Lemma~5.1 and Definition~5.2]{AlfShultz03}, \cite{Neal} or \cite[Proposition 2]{FriRu85}).
Therefore, $\|.\|_{\phi}$ induces a metric on the norm closed
unit ball of $Q(p)(M)$, which gives a topology homeomorphic to the strong$^*$-topology of $M$ restricted to
the closed unit ball of $Q(p)(M)$.

\begin{remark}\label{r support control functional for sequences}{
Let $M$ be a JBW$^*$-algebra.
Every bounded sequence $(\varphi_k)$ of positive functionals in the closed unit ball of $M_*$ essentially lies in the closed unit ball of the predual of a JBW$^*$-subalgebra $M_0$ admitting a faithful normal state $\phi_0$ with the following property: for each $a\geq 0$ in $M$, $\phi_0 (a) =0$ if, and only if, $\varphi_k (a) =0$, for every $k$.
Indeed, let $\phi_0 := \sum_{k=1}^{\infty} \frac{1}{2^{k}} \varphi_{k}\in M_*$. Clearly, $\phi_0$ is a positive normal functional on $M$ satisfying the above property (that is, according to the terminology employed in \cite{BrookSaWri}, $(\varphi_k)$ is absolutely continuous with respect to $\phi_0$). Let $p_0$ denote the support projection of $\phi_0$ in $M$. Then $\phi_0$ is faithful on $M_0:=Q(p_0) (M)$ and the strong$^*$-topology of $M_0$ is metrised by the Hilbertian seminorm $\|.\|_{\phi_0}$ on bounded sets of $M_0$. Since, for each natural $k$, $\varphi_k (1-p_0) =0$, we deduce that $\|\varphi_k\| = \varphi_k (1) = \varphi_k (p_0)$, and hence $\varphi_k = \varphi_k P_2(p_0)$ {\rm(}cf. \cite[Proposition 1]{FriRu85}{\rm)}.\smallskip

The sequence $(\varphi_k)$ is said to be supported by the JBW$^*$-subalgebra $M_0$ when the above conditions are satisfied.}
\end{remark}

Let $A$ be a C$^*$-algebra. If $x$ and $y$ are positive elements in $A$
such that $x \geq y$, then $x^{\alpha}\geq y^{\alpha}$ for any $0 \leq \alpha \leq 1$ (cf. \cite[Proposition I.6.3]{Tak}).
A similar statement is, in general, false for $\alpha>1$.
Now, let $J$ be a JB$^*$-algebra. Since, by the Shirshov-Cohn theorem (cf. \cite[7.2.5]{HancheStor}), the JB$^*$-subalgebra
of $J$ generated by two positive elements is a JC$^*$-algebra, the above operator monotonicity behavior also holds in $J$.

\begin{lemma}
\label{l 3.4} Let $M$ be a JBW$^*$-algebra. Let $(\varphi_k)$ be a sequence of positive functionals in the closed unit ball of $M_*$ and let $p_0$ be a projection in $M$ such that $M_0= Q(p_0) (M)$ is a JBW$^*$-algebra which admits a faithful normal state $\phi_0$ and $(\varphi_k)$ is supported by $M_0$.
Let $(a_n)$ be a strong$^*$-null, decreasing sequence in $M$ satisfying $0\leq a_n\leq 1$, for every $n$.
\begin{enumerate}[$(a)$]\item If $\displaystyle{\liminf_{n\to \infty}} \sup_{k} \| \varphi_{k} Q({a_n}) \| \geq\gamma>0$, then for every $\varepsilon >0$ there exists a strong$^*$-null, decreasing sequence $(p_n)$ of projections in $M$ such that \linebreak$p_n \leq r(Q(p_0)a_1)$ {\rm($\forall n$)} and $\displaystyle \lim_{n\to \infty} \sup_{k} \| \varphi_{k} Q(p_n) \| \geq \gamma - \varepsilon$.
\item If $\displaystyle{\liminf_{n\to \infty}} \| \varphi_{n} Q({a_n}) \| \geq\gamma>0$, then for every $\varepsilon >0$ there exist a subsequence $(\varphi_{\sigma(k)})$ and a strong$^*$-null, decreasing sequence $(p_n)$ of projections in $M$ satisfying $p_n \leq r(Q(p_0)a_1)$ {\rm($\forall n$)} and $\displaystyle \liminf_{n\to \infty} \| \varphi_{\sigma(n)} Q(p_n) \| \geq \gamma - \varepsilon$.
\end{enumerate}
Furthermore, when $M$ admits a faithful normal state, we can take $p_0=1$.
\end{lemma}

\begin{proof}

$(a)$ Let us observe that, for each $n$ and $k$, $$\| \varphi_{k} Q({a_n}) \|= \varphi_{k} Q({a_n}) (1) = \varphi_{k} Q(p_0) Q({a_n}) (1) $$
$$= \varphi_{k} Q(p_0) ({a_n^2})\leq \varphi_{k} Q(p_0) ({a_n}) $$ where $(Q(p_0) ({a_n}))$ defines a strong$^*$-null, decreasing sequence of positive elements in $M_0$. Let $c_n\in M_0$ denote the square root of $Q(p_0) (a_n)$ (that is, $c_n^2 = Q(p_0) (a_n)$). It is not hard to check that $(c_n)$ is a strong$^*$-null sequence of positive elements in the closed unit ball of $M_0$. By hypothesis, $(c_n^{2})$ is a decreasing sequence, and hence $(c_n)$ also is decreasing. We additionally know that
$$ {\liminf_{n\to \infty}} \sup_{k} \| \varphi_{k} Q({c_n}) \| = {\liminf_{n\to \infty}} \sup_{k} \varphi_{k} ({c_n^2}) $$
$$= {\liminf_{n\to \infty}} \sup_{k} \varphi_{k} Q(p_0) ({a_n}) \geq {\liminf_{n\to \infty}} \sup_{k} \| \varphi_{k} Q({a_n}) \| \geq\gamma.$$

Take a sequence $(\varepsilon_j)\subset (0,1)$ such that $\sum_{j=1}^\infty \varepsilon_j = \varepsilon.$ Corollary \ref{c l 3.3}, applied to $(c_n)$, $(\varphi_k)$, $\gamma$ and $\varepsilon_1$, assures the existence of a strong$^*$-null sequence of projections $(q_n^{(1)})$ in $M_0$ satisfying $q_n^{(1)}\leq q_1^{(1)}\leq r (c_1) \leq r (Q(p_0) a_1)$ for every $n\geq 1$ and,
for each $n_0$, $${\liminf_{n\to \infty}} \sup_{k} \| \varphi_{k} Q(Q(q_{n_0}^{(1)})(c_n)) \| \geq \gamma - \varepsilon_1.$$ Take $n_1 \geq 1$ such that $\|q_{n_1}^{(1)}\|_{\phi_0} \leq \frac12$ and define $c_n^{(2)} := Q(q_{n_1}^{(1)})(c_n)$. Applying Lemma \ref{l 3.3} on $(c_n^{(2)})$, $(\varphi_k)$, $\gamma-\varepsilon_1$ and $\varepsilon_2$, we deduce the existence of a strong$^*$-null sequence of projections $(q_n^{(2)})$ in $M_0$ satisfying $q_n^{(2)}\leq q_1^{(2)}\leq q_{n_1}^{(1)}$ for every $n\geq 1$ and,
for each $n_0$, $${\liminf_{n\to \infty}} \sup_{k} \| \varphi_{k} Q(Q(q_{n_0}^{(2)})\ c_n^{(2)}) \|={\liminf_{n\to \infty}} \sup_{k} \| \varphi_{k} Q(Q(q_{n_0}^{(2)})\ c_n) \| \geq \gamma - \sum_{i=1}^{2} \varepsilon_i.$$ 

By induction, there exists an strictly increasing sequence $(n_j)$ in $\mathbb{N}$ and sequences of projections $(q_{n}^{(j)})$ in $M_0$ satisfying
$q_n^{(j)}\leq q_1^{(j)}\leq q_{n_{j-1}}^{(j-1)}$ for every $n\geq 1$, $\|q_{n_j}^{(j)}\|_{\phi_0} \leq \frac{1}{2^j}$ and
$${\liminf_{n\to \infty}} \sup_{k} \| \varphi_{k} Q(Q(q_{n_j}^{(j)})(c_n)) \| \geq \gamma - \sum_{i=1}^{j} \varepsilon_i.$$

Finally, defining $p_j:= q_{n_{j}}^{(j)}$, we obtain a strong$^*$-null, decreasing sequence of projections in $M_0$ such that $${\liminf_{n\to \infty}} \sup_{k} \| \varphi_{k} Q(Q(p_j)(c_n)) \| \geq \gamma - \sum_{i=1}^{j} \varepsilon_i.$$ Therefore, there exists a subsequence $(c_{n_j})$ such that $$\sup_{k} \| \varphi_{k} Q(p_j)\|\geq \sup_{k} \| \varphi_{k} Q(p_j)Q(c_{n_j})Q(p_j) \| =$$ $$\sup_{k} \| \varphi_{k} Q(Q(p_j)(c_{n_j})) \|  \geq \gamma - \sum_{i=1}^{j+1} \varepsilon_i >  \gamma -  \varepsilon,$$ proving the first statement.\smallskip

$(b)$ Follows from $(a)$ and Lemma \ref{l 2.4}$(b)$.
\end{proof}

The set of all strong$^*$-null decreasing sequences in a von Neumann algebra plays a key role in \cite{Randri}. Following the same notation,
here the symbol $\mathcal{D}$ will stand for the set of all strong$^*$-null decreasing sequences
of projections in a JBW$^*$-algebra $M$.

\begin{corollary}\label{c common support projection} Let $(\varphi_k)$ be a sequence of positive functionals in the closed unit ball of the predual of a JBW$^*$-algebra $M$. Let $p_0$ be a projection in $M$ such that $M_0= Q(p_0) (M)$ is a JBW$^*$-algebra which admits a faithful normal state $\phi_0$ and $(\varphi_k)$ is supported by $M_0$. Then {\small $${\rm\sup}\left\{\lim_{n} \sup_{k} \|\varphi_{k} Q(p_n)\|: (p_n)\in \mathcal{D} \right\} = {\rm\sup}\left\{\lim_{n} \sup_{k} \|\varphi_{k} Q(q_n)\|: \begin{array}{c}
                                                                       (q_n)\in \mathcal{D} \\
                                                                       (q_n)\subset M_0
                                                                     \end{array} \right\}.$$}
\end{corollary}

\begin{proof} To simplify notation let $S_1$ and $S_2$ denote the supreme in the left and right hand side, respectively.
The inequality $S_1 \geq S_2$ is clear. In order to prove the reciprocal inequality let us fix an arbitrary $\varepsilon>0$ and $(p_n)\in \mathcal{D}$ with $\displaystyle \lim_{n} \sup_{k} \|\varphi_{k} Q(p_n)\| > S_1 -\varepsilon$. Let us denote $b_n:=Q(p_0) (p_n)\in M_0$ and $c_n:= b_n^{\frac12}$ (Clearly, $0\leq b_n, c_n \leq 1$). The sequence $(b_n)\subset M_0$ is decreasing and strong$^*$-null, and thus $(c_n)$ satisfies the same properties.
Since $$\|\varphi_{k} Q(p_n)\| = \varphi_{k}(p_n) = \varphi_{k} Q(p_0) (p_n)= \varphi_{k} (b_n) = \varphi_{k} (c_n^2) = \|\varphi_k Q(c_n)\|,$$ we can deduce that   $\displaystyle \liminf_{n} \sup_{k} \|\varphi_{k} Q(c_n)\| = \lim_{n} \sup_{k} \|\varphi_{k} Q(p_n)\| > S_1 -\varepsilon$. By Lemma \ref{l 3.4}$(a)$, there exists a strong$^*$-null, decreasing sequence $(q_n)$ of projections in $M$ satisfying $q_n \leq r(c_1) \leq p_0$ (in particular, $(q_n)\subset M_0$) and $\displaystyle \lim_{n\to \infty} \sup_{k} \| \varphi_{k} Q(q_n) \| \geq S_1 - 2 \varepsilon$. 
\end{proof}

\begin{proposition}\label{p 3.5} Let $K$ be a bounded subset in the positive part of the predual of a JBW$^*$-algebra $M$.
Suppose that $K$ is not relatively weakly compact. Then there exists a sequence $(\varphi_k)$ in $K$ satisfying{\small
$${\rm\sup}\left\{\lim_{n} \sup_{k} \|\varphi_{k} Q(p_n)\|: (p_n)\in \mathcal{D} \right\}={\rm \sup} \Big\{{\liminf_{m}} \|\varphi_{m} Q(p_m)\|: (p_m)\in \mathcal{D} \Big\} >0.$$}
\end{proposition}

\begin{proof} It is easy to see that $${\rm\sup}\left\{\lim_{n} \sup_{k} \|\varphi_{k} Q(p_n)\|: (p_n)\in \mathcal{D} \right\}\geq {\rm \sup} \Big\{{\liminf_{m}} \|\varphi_{m} Q(p_m)\|: (p_m)\in \mathcal{D} \Big\},$$ for every sequence $(\varphi_{n})$ in $K$.\smallskip

We may assume that $K\subset B(M_*)$. Since $K$ is not relatively weakly compact, the classical
theorems of Eberlein-${\rm \breve{S}}$mul'jan and Rosenthal assure the existence of a sequence
$(\phi_k)$ in $K$ which is isomorphically equivalent to the unit vector basis of $\ell^{1}$.
Denote $\displaystyle\alpha_0 :=\sup\left\{\lim_{n}\sup_{k} \|\phi_k Q(p_n)\|: (p_n)\in \mathcal{D} \right\}$.
By Proposition~\ref{p equiv rwc and uniform integrable}, we have $\alpha_0 >0$.\smallskip

Fix a sequence $(\varepsilon_j)$ in $(0,1)$, such that $\prod_{j} (1-\varepsilon_j) >0$. Find $(q_n)\in \mathcal{D}$
such that $\displaystyle\lim_{n}\sup_{k} \|\phi_k Q(q_n)\| \geq \alpha_0 (1-\varepsilon_1)$. By Lemma \ref{l 2.4}, there exist subsequences
$(\phi_n^{(1)}) = (\phi_{k_n})$ and $(p_m^{(1)}) = (q_{n_{m}})$ satisfying $$\exists \lim_{m\to \infty} \| \phi_{{m}}^{(1)} Q({p_{{m}}^{(1)}})\| \geq \alpha_0 (1-\varepsilon_1).$$
Set $\displaystyle\alpha_1 :=\sup\left\{\lim_{n}\sup_{m} \|\phi^{(1)}_m Q(p_n)\|: (p_n)\in \mathcal{D} \right\}$. Clearly $\alpha_1\geq \alpha_0 (1-\varepsilon_1).$ applying again Lemma \ref{l 2.4} we can find, by definition of $\alpha_1$, a subsequence $(\phi_n^{(2)}) = (\phi^{(1)}_{k_n})$ and a sequence $(p_m^{(2)})$ in $\mathcal{D}$ such that $$\exists \lim_{m\to \infty} \| \phi_{{m}}^{(2)} Q({p_{{m}}^{(2)}})\| \geq \alpha_1 (1-\varepsilon_2)\geq  \alpha_0 \prod_{j=1}^{2} (1-\varepsilon_j).$$ We can inductively choose subsequences $(\phi_n^{(1)}) \supseteq (\phi_n^{(2)}) \supseteq \ldots \supseteq (\phi_n^{(j)}) \supseteq \ldots$ and sequences $(p_n^{(1)}), \ldots, (p_n^{(j)}), \ldots$ in $\mathcal{D}$ satisfying $$\exists \lim_{m\to \infty} \| \phi_{{m}}^{(j)} Q({p_{{m}}^{(j)}})\| \geq \alpha_{j-1} (1-\varepsilon_j)\geq  \alpha_0 \prod_{i=1}^{j} (1-\varepsilon_i),$$ where $\displaystyle\alpha_j :=\sup\left\{\lim_{n}\sup_{m} \|\phi^{(j)}_m Q(p_n)\|: (p_n)\in \mathcal{D} \right\}$.\smallskip

Now, define $\varphi_{n} :=\phi_n^{(n)}$. Since $(\varphi_n)_{n\geq j}$ is a subsequence of $(\phi_n^{(j)})$ and hence
$$\| \varphi_n Q(p_n^{(j)})\| = \varphi_n (p_n^{(j)})= \phi_{\sigma(n)}^{(j)} (p_n^{(j)})  $$ $$\geq \phi_{\sigma(n)}^{(j)} (p_{\sigma(n)}^{(j)}) = \| \phi_{\sigma(n)}^{(j)} Q(p_{\sigma(n)}^{(j)}) \| \ (n\geq j),$$ it follows that
$$\liminf_{n} \| \varphi_n Q(p_n^{(j)})\| \geq \liminf_{n} \| \phi_{\sigma(n)}^{(j)} Q(p_{\sigma(n)}^{(j)}) \| $$
$$= \lim_{n} \| \phi_{\sigma(n)}^{(j)} Q(p_{\sigma(n)}^{(j)}) \| \geq \alpha_{j-1} (1-\varepsilon_j).$$ In particular,
$\displaystyle\liminf_{n} \| \varphi_n Q(p_n^{(1)})\| \geq \alpha_{0} (1-\varepsilon_1) >0.$
A similar argument shows that $$\beta= \sup\left\{\lim_{n}\sup_{k} \|\varphi_k Q(p_n)\|: (p_n)\in \mathcal{D} \right\} \leq \alpha_{j},$$ for every $j$. The inequality $\displaystyle\frac{1}{1-\varepsilon_j} \liminf_{n} \| \varphi_n Q(p_n^{(j)})\| \geq \alpha_{j-1} \geq \beta$, implies that $\displaystyle\beta \leq \frac{1}{1-\varepsilon_j} {\rm \sup} \Big\{{\liminf_{m}} \|\varphi_{m} Q(p_m)\|: (p_m)\in \mathcal{D} \Big\}$. Taking limit in $j$, we assert that $${\rm\sup}\left\{\lim_{n} \sup_{k} \|\varphi_{k} Q(p_n)\|: (p_n)\in \mathcal{D} \right\}\leq {\rm \sup} \Big\{{\liminf_{m}} \|\varphi_{m} Q(p_m)\|: (p_m)\in \mathcal{D} \Big\}.$$
\end{proof}

The following technical lemma will be required later.

\begin{lemma}\label{l hidden 1} Let $p$ be a projection in a JBW$^*$-algebra $M$, $T:M\to M$ a positive bounded linear operator, and let $\varphi$ be a positive (normal) functional on $M$ with $\|\varphi\|\leq 1$. Then the following inequalities hold for each $a$ in $M$ with $0\leq a\leq 1$: \begin{enumerate}[$(a)$]
\item $\displaystyle \|\varphi Q(p) Q(a) \|^2 \leq  \|\varphi Q(p) Q(a) Q(p)\|$;
\item $\left|\varphi Q(1,p) T (x) \right|^{2}\leq \|T\| \varphi Q(p) T (x)$, for every $0\leq x\leq 1,$ and hence, $$\left\|\varphi Q(1,p) T \right\|\leq 2 \|T\|^{\frac12} \left\|\varphi Q(p) T\right\|^{\frac12};$$
\item $\|\varphi Q(1-p,p) Q(a) \| \leq {2}\ \|\varphi Q(p) Q(a) Q(p)\|^{\frac14}+ \|\varphi Q(p) Q(a) Q(p)\|^{\frac12}.$
\end{enumerate}

\end{lemma}

\begin{proof}$(a)$ Let $\varphi$, $a$ and $p$ be as in the hypothesis of the lemma. Since $\varphi Q(p) Q(a)$
and $\varphi Q(p) Q(a) Q(p)$ are positive functionals on $M$, it follows, by the Cauchy-Schwarz inequality, that
$$\|\varphi Q(p) Q(a) \|^2 = \Big(\varphi Q(p) Q(a) (1)\Big)^2 = \Big(\varphi Q(p) (a^2)\Big)^2 $$
$$\leq \Big( \varphi Q(p)(a) \Big)^2 \leq \varphi \Big(Q(p)(a)\circ Q(p)(a)\Big)= \varphi Q( Q(p)(a) ) (1) $$
$$= \|\varphi Q( Q(p)(a) ) \| = \|\varphi Q(p) Q(a) Q(p)\|. $$

$(b)$ First, we observe that $Q(1,p) (z) = p\circ z,$ for every $z$ in $M$. Fix $x$ in $M$ with $0\leq x\leq 1$.
Since, by the Shirshov-Cohn theorem (see \cite[7.2.5]{HancheStor}), the JB$^*$-subalgebra $J$ generated by $p$, $x$ and
the unit is a JC$^*$-algebra, we may assume that $J$ is a JC$^*$-subalgebra of a C$^*$-algebra $A$.
The functional $\varphi|_{J}$ extends to a positive functional of $A$, denoted again by $\varphi$ (cf. \cite[Proposition]{Bun01}).
Since $\varphi$ is positive, $\varphi (x p) = \overline{\varphi ((x p)^*)} = \overline{\varphi (p x)}$. In this case we have
$$\left|\varphi Q(1,p) (x) \right|^{2} = \left|\varphi(p\circ x)\right|^2 = \frac14 \Big|\varphi(p x+ x p)\Big|^2$$
$$\leq |\varphi(p x)|^2 \leq \varphi(p x x p) \leq \varphi(p x p) = \varphi Q(p) (x).$$ Since $x$ was arbitrarily chosen, the first inequality follows by replacing $x$ with $\frac{1}{\|T\|} T (x)$. The second statement is a direct consequence of the first one.\smallskip

$(c)$ Combining $(a)$ and $(b)$ with $T=Q(a)$, we get $$\|\varphi Q(1,p) Q(a) \| \leq 2\ \|\varphi Q(p) Q(a) \|^{\frac12} \leq 2\ \|\varphi Q(p) Q(a) Q(p)\|^{\frac14},$$ which assures that $$\|\varphi Q(1-p,p) Q(a) \|\leq \|\varphi Q(1,p) Q(a) \| + \|\varphi Q(p) Q(a) \|$$ $$ \leq 2\ \|\varphi Q(p) Q(a) Q(p)\|^{\frac14}+ \|\varphi Q(p) Q(a) Q(p)\|^{\frac12}.$$
\end{proof}

Following standard notation, the symbols $\bigvee$ and $\bigwedge$ to denote the
least upper bound and greatest lower bound in the set of projections in $M$. The
existence of the least upper bound and greatest lower bound of a family of projections is guaranteed by
\cite[Lemma 4.2.8]{HancheStor}.

\begin{lemma}\label{l operator commuting supreme and sum} Let $(q_n)$ be a sequence of projections in a JBW$^*$-algebra $M$
whose elements mutually operator commute. Let $\phi$ be a positive normal functional on $M$ such that
$\sum_{n\geq 1} \phi (q_n) <\infty$. Then defining $p_n= \bigvee_{k\geq n} q_k$, it follows that $\phi (p_n) \to 0$.
\end{lemma}

\begin{proof} It follows from the definition of operator commutativity that the $p_n$'s generate
an associative commutative JBW$^*$-subalgebra which is JBW$^*$-isomorphic to
an abelian von Neumann algebra of the form $C(K)$, where $K$ is a Stonean compact Hausdorff
space (compare \cite[Proposition 2.11]{AlfShultz03}). The rest of the proof is an exercise left to the reader.
\end{proof}

Now we have the tools too prove a Jordan version of \cite[Theorem 3.1]{Randri}.

\begin{theorem}\label{t 3.1} Let $K$ be a bounded subset in the positive part of the predual of a JBW$^*$-algebra $M$.
Then there there exists a sequence $(q_n)$ in
$\mathcal{D}$ and a sequence
$(\phi_k)$ in $K$ satisfying
$$\sup\left\{\lim_{n}\sup_{k\in \mathbb{N}} \|\phi_k Q(p_n)\|: (p_n)\in \mathcal{D} \right\} = \lim_{n}\sup_{k\in \mathbb{N}} \|\phi_k Q(q_n)\|.$$
Furthermore, there exists a projection $q_0$ in $M$ such that $M_0= Q(q_0) (M)$ is a JBW$^*$-algebra which admits a faithful normal state, $(\varphi_k)$ is supported by $M_0$ and the above supreme coincides with $$ \sup\left\{\lim_{n}\sup_{k\in \mathbb{N}} \|\phi_k Q(p_n)\|: (p_n)\in \mathcal{D}, (p_n)\subseteq M_0 \right\} .$$
\end{theorem}

\begin{proof} When $K$ is relatively weakly compact the supremum is zero, and hence attained at every $ (p_n)\in \mathcal{D}$ (cf. Proposition \ref{p equiv rwc and uniform integrable}). We may therefore assume that $K$ is not relatively weakly compact.\smallskip

Arguing by contradiction, we assume that for every sequence $(\varphi_k)$ in $K$, the supreme
$\displaystyle\sup\left\{\lim_{n}\sup_{k\in \mathbb{N}} \|\varphi_k Q(p_n)\|: (p_n)\in \mathcal{D} \right\}$ is not attained.\smallskip

We may assume that $K\subset B(M_*)$. By assumptions, $K$ is not relatively weakly compact, so from the Eberlein-${\rm \breve{S}}$mul'jan and Rosenthal
theorems there exists a sequence $(\psi_k)$ in $K$ which is isomorphically equivalent to the unit vector basis of $\ell^{1}$.\smallskip

Proceeding as in Remark \ref{r support control functional for sequences}, we set $\displaystyle \phi_0 := \sum_{k=1}^{\infty} \frac{1}{2^{k}} \psi_{k}\in M_*$ and $q_0= s(\phi_0)$ the support projection of $\phi_0$. Then $\phi_0$ is faithful on $M_0:=Q(q_0) (M)$, the strong$^*$-topology of $M_0$ is metrised by the Hilbertian seminorm $\|.\|_{\phi_0}$ on bounded sets of $M_0$, and the sequence $(\psi_k)$ is supported by $M_0$ (in particular, $\psi_k = \psi_k P_2(q_0)$). Since the set $\widetilde{K}:=\{\psi_k|_{M_0} : k\in \mathbb{N}\}$ is not relatively weakly compact in $(M_0)_{*},$ by Proposition \ref{p 3.5} and Corollary \ref{c common support projection}, there exists a sequence $(\varphi_k)$ in $\widetilde{K}\subset K$ satisfying
\begin{equation}\label{eq attaining} \alpha= {\rm\sup}\left\{\lim_{n} \sup_{k} \|\varphi_{k} Q(p_n)\|: (p_n)\in \mathcal{D},\ (p_n)\subset M_0 \right\}
\end{equation} $$= {\rm \sup} \Big\{{\liminf_{m}} \|\varphi_{m} Q(p_m)\|: (p_m)\in \mathcal{D}, (p_m)\subset M_0 \Big\} $$
$$={\rm\sup}\left\{\lim_{n} \sup_{k} \|\varphi_{k} Q(p_n)\|: (p_n)\in \mathcal{D} \right\}>0.$$
By assumption, $\alpha$ is not attained. We shall get a contraction working on $M_0$. Henceforth, we assume $M=M_0$ and $\phi_0$ is a faithful normal functional on $M$.\smallskip

Applying an induction argument, we can find sequences $(m_j)$ and $(n_j)$ in $\mathbb{N}$ with $(n_j)$ strictly increasing, subsequences $(\varphi_n)= (\varphi_n^{(1)}) \supseteq (\varphi_n^{(2)}) \supseteq \ldots \supseteq (\varphi_n^{(j)}) \supseteq \ldots$, and sequences $(s_n^{(j)})_{n},$ $(S_n^{(j)})_{n},$ $(p_n^{(1)})$, $(p_n^{(2)}), \ldots (p_n^{(j)}), \ldots$ in $\mathcal{D}$ satisfying the following properties:
\begin{equation}\label{eq induction one} p_n^{(j+1)} \perp S^{(j)}_{m}, \ \forall n,m\in \mathbb{N} \hbox{ with } m\geq n_{j},
\end{equation}
\begin{equation}\label{eq induction two} \alpha (1-\frac{1}{2^{m_j-1}}) \leq  \lim_{n} \sup_{k\in \mathbb{N}} \| \varphi_k Q(S_n^{(j)})\| \left(< \alpha (1-\frac{1}{2^{m_j}})\right),
\end{equation}
\begin{equation}\label{eq induction two half} \sup_{k\in \mathbb{N}} \| \varphi_k Q(S_{n_j}^{(j)})\| < \alpha (1-\frac{1}{2^{m_j}}),
\end{equation} and
\begin{equation}\label{eq induction three} \liminf_{n} \|\varphi_{n}^{(j+1)} Q(s_n^{(j+1)})\| \geq  \liminf_{n} \|\varphi_{n}^{(j)} Q(s_n^{(j)})\| + \frac{\alpha^{4}}{(2^{m_j+5})^4},
\end{equation} where $(s_n^{(1)})=(S_n^{(1)})_{n}=(p_n^{(1)})_{n}$ and for $j\in \mathbb{N}$, $(s_n^{(j+1)})$ and $(S_n^{(j+1)})_{n}$ are the sequence in $\mathcal{D}$ given by
$$s_n^{(j+1)} := \left\{\begin{array}{lc}
    1  & ; n<n_{j} \\
    s_{\sigma_j(n)}^{(j)} + p_{n}^{(j+1)}  & ; n\geq n_j,
    \end{array} \right.
S_n^{(j+1)} := \left\{\begin{array}{lc}
    1  & ; n<n_{j} \\
    S_{n}^{(j)} + p_{n}^{(j+1)}  & ; n\geq n_j,
    \end{array} \right.
    $$ and $\sigma_j: \mathbb{N}\to \mathbb{N}$ is the mapping satisfying $(\varphi_{\sigma_j(n)}^{(j)}) = (\varphi_n^{(j+1)})$.\smallskip

Since $\alpha$ is not attained, we can find $(p_n^{(1)})\in \mathcal{D}$ and $m_1$ in $\mathbb{N}$ satisfying $$\alpha (1-\frac{1}{2^{m_1-1}}) \leq \lim_{n} \sup_{k} \|\varphi_{k} Q(p^{(1)}_n)\| < \alpha (1-\frac{1}{2^{m_1}}).$$ There exists $n_1 >1$ such that $$\sup_{k} \|\varphi_{k} Q(p^{(1)}_{n_1})\| < \alpha (1-\frac{1}{2^{m_1}}).$$

Suppose now the first $j$ elements and sequences have been defined.
By hypothesis (see $(\ref{eq attaining})$), there exists $(\widetilde{q}_m)$ in $\mathcal{D}$ with $$\displaystyle{\liminf_{n}} \|\varphi_{n} Q(\widetilde{q}_n)\| > \alpha (1-\frac{1}{2^{m_j+1}}).$$ Since $ (\varphi_m^{(j)}) = (\varphi_{n_m})$ is a subsequence of $ (\varphi_n)$, we have $$\liminf_{m} \|\varphi_{m}^{(j)} Q(\widetilde{q}_{n_m})\| \geq {\liminf_{n}} \|\varphi_{n} Q(\widetilde{q}_n)\| > \alpha (1-\frac{1}{2^{m_j+1}}).$$

If we write $q_m:= \widetilde{q}_{n_m}$ and $$\varphi_{m}^{(j)} Q(q_m)  = \varphi_{m}^{(j)} Q(S^{(j)}_{n_j}) Q(q_m) + 2 \varphi_{m}^{(j)} Q(1-S^{(j)}_{n_j},S^{(j)}_{n_j}) Q(q_m) $$ $$+ \varphi_{m}^{(j)} Q(1-S^{(j)}_{n_j}) Q(q_m).$$ Now, recalling that $Q(a) Q(b) Q(a) = Q(Q(a)b),$ we deduce, by Lemma \ref{l hidden 1} $(a)$ and $(c)$, that $$\|\varphi_{m}^{(j)} Q(q_m) \| \leq \| \varphi_{m}^{(j)} Q(S^{(j)}_{n_j})\| + 6 \| \varphi_{m}^{(j)} Q( Q(1-S^{(j)}_{n_j})(q_m)) \|^{\frac14} $$ $$+ \| \varphi_{m}^{(j)} Q( Q(1-S^{(j)}_{n_j})(q_m)) \|^{\frac12}.$$ Since $K\subset B(M_*)$, we have, by $(\ref{eq induction two half})$, $$7 \| \varphi_{m}^{(j)} Q( Q(1-S^{(j)}_{n_j})(q_m)) \|^{\frac14} \geq \|\varphi_{m}^{(j)} Q(q_m) \| - \| \varphi_{m}^{(j)} Q(S^{(j)}_{n_j})\| $$ $$\geq \|\varphi_{m}^{(j)} Q(q_m) \| - \alpha (1-\frac{1}{2^{m_j}}),$$ which implies that
$${\liminf_{m}}  \| \varphi_{m}^{(j)} Q( Q(1-S^{(j)}_{n_j})(q_m)) \|^{\frac14} $$ $$\geq \frac17 \left( {\liminf_{m}} \|\varphi_{m}^{(j)} Q(q_m)\| - \alpha (1-\frac{1}{2^{m_j}})\right)> \frac{\alpha}{2^{m_j+4}}.$$

We define $a_m^{(j)} := Q(1-S^{(j)}_{n_j})(q_m).$ Clearly, $(a^{(j)}_m)$ is a strong$^*$-null, decreasing sequence with $0\leq a_m^{(j)} \leq 1$ and $\displaystyle{\liminf_{m}}  \| \varphi_{m}^{(j)} Q( a_m^{(j)}) \| > \frac{\alpha^4}{(2^{m_j+4})^4}.$ By Lemma \ref{l 3.4}$(b)$, applied to $(\varphi_{m}^{(j)})$ and $( a_m^{(j)})$ $\Big($with $\varepsilon = \frac{\alpha^{4}}{(2^{m_j+4})^4} -\frac{\alpha^{4}}{(2^{m_j+5})^4}\Big)$,
there exists $(p_n^{(j+1)})\in \mathcal{D}$ with $p_1^{(j+1)} \leq r(a_1^{(j)})\leq 1- S_{n_j}^{(j)}$ and a subsequence $(\varphi_{m}^{(j+1)})= (\varphi_{\sigma_j(m)}^{(j)})$ such that
$$ \liminf_{n} \|\varphi_{n}^{(j+1)} Q(p_n^{(j+1)})\| \geq  \frac{\alpha^{4}}{(2^{m_j+5})^4}.$$

Observe that since $p_1^{(j+1)} \leq  1- S_{n_j}^{(j)}$ we have $p_n^{(j+1)} \perp  S_{m}^{(j)}$ for every $n,m\in \mathbb{N}$ with $m\geq n_j$.
Since, by hypothesis, $\displaystyle 0\leq \lim_n \sup_k \|\varphi_k Q(S_n^{(j+1)})\| < \alpha$,
we can choose convenient $m_{j+1}$ and $n_{j+1}>n_{j}$ satisfying $(\ref{eq induction two})$ and $(\ref{eq induction two half}).$\smallskip

It is also clear that $s_n^{(j)}\leq S_n^{(j)}$, and thus, for each $n\geq n_j$, $s_n^{(j)} + p_n^{(j+1)}$
is a projection in $M$ satisfying that $$\liminf_{n}  \| \varphi_n^{(j+1)} Q(s_n^{(j+1)})\| = \liminf_{n} \| \varphi_n^{(j+1)} Q(s_{\sigma_j(n)}^{(j)} + p_n^{(j+1)})\|$$
$$\geq \liminf_{n} \| \varphi_n^{(j+1)} Q(s_{\sigma_j(n)}^{(j)})\| + \liminf_{n} \| \varphi_n^{(j+1)} Q( p_n^{(j+1)})\| $$
$$\geq \liminf_{n} \| \varphi_{\sigma_j(n)}^{(j)} Q(s_{\sigma_j(n)}^{(j)})\| + \frac{\alpha^{4}}{(2^{m_j+5})^4}\geq
\liminf_{n} \| \varphi_n^{(j)} Q(s_n^{(j)})\| + \frac{\alpha^{4}}{(2^{m_j+5})^4},$$ which proves $(\ref{eq induction three})$
and concludes the induction argument.\smallskip

We shall now show that $\displaystyle \lim_j m_j =+\infty$. By $(\ref{eq induction three})$, $$\liminf_{n}  \| \varphi_n^{(j)} Q(s_n^{(j)})\| \geq \liminf_{n}  \| \varphi_n^{(1)} Q(s_n^{(1)})\| + \sum_{k=1}^{j-1} \frac{\alpha^{4}}{(2^{m_k+5})^4},$$ witnessing the desired statement.\smallskip

For each $n\geq n_j> n_{j-1} >\ldots > n_1$, we have $$S_{n}^{(j)} = S_{n}^{(j-1)} + p^{(j)}_{n} =S_{n}^{(j-2)} + p^{(j-1)}_{n} + p^{(j)}_{n}=\ldots = \sum_{k=1}^{j} p_{n}^{(k)},$$ where the summands appearing in the last sum are mutually orthogonal. We claim that $S_{n}^{(j)}$ and $S_{m}^{(l)}$ operator commute whenever $n\geq n_j$, $m\geq n_l$, $j\geq l$ and $n\geq m$. Indeed, since $\displaystyle S_{n}^{(j)} = \sum_{k=1}^{j} p_{n}^{(k)}$, $\displaystyle s_{m}^{(l)} = \sum_{k=1}^{l} p_{m}^{(k)} $, where the summands in each sum are mutually orthogonal, we get for each $k_1 \neq k_2, l$ in $\{1,\ldots,l\}$, $$p_{n}^{(k_1)} \leq p_{m}^{(k_1)} \perp p_{m}^{(k_2)}, p^{(l)}_{m}.$$ On the other hand, $p_n^{(k_1)}$ has been chosen to be orthogonal to $S^{(k_2)}_{m}$ for every $k_1,k_2,n,$ and $m$ in $\mathbb{N}$ with $k_1>k_2$, $m\geq n_{k_2}$. Therefore, $$\J {S_{m}^{(l)}}{S_{n}^{(j)}}{S_{m}^{(l)}} = 2 (S_{m}^{(l)}\circ S_{n}^{(j)}) \circ S_{m}^{(l)} - S_{m}^{(l)} \circ S_{n}^{(j)} = \sum_{k=1}^{l} p_{n}^{(k)} = (S_{m}^{(l)})^2 \circ S_{n}^{(j)}, $$ proving the claim.\smallskip

By Lemma \ref{l 2.4}$(a)$ and $(\ref{eq induction two})$, we can find strictly increasing sequences $(l_j)_j$ and $(k_j)_j$ in $\mathbb{N}$ with $k_j > n_j$, $\displaystyle\|S_{k_j}^{(j)}\|_{\phi_0} <\frac{1}{2^j}$ and $$\displaystyle\alpha \left(1-\frac{1}{2^{m_j-2}}\right) < \| \varphi_{l_j} Q(S_{k_j}^{(j)})\|.$$ The projections in the sequence $(S_{k_j}^{(j)})$ mutually operator commute, so defining $\widehat{p}_{j} := \bigvee_{m\geq j} S_{k_m}^{(m)}$ we get a decreasing sequence of projections in $M$ which, by Lemma \ref{l operator commuting supreme and sum}, is strong$^*$-null. \smallskip

We define $\phi_j := \varphi_{l_j}$.  Then $$\| \phi_j Q(\widehat{p}_j)\| = \phi_j (\widehat{p}_j) \geq \phi_j (S_{k_j}^{(j)}) = \| \phi_j Q (S_{k_j}^{(j)}) \| >\alpha \left(1-\frac{1}{2^{m_j-2}}\right),$$ which implies that $$\lim_{n} \sup_{k} \| \phi_k Q(\widehat{p}_n)\| \geq\liminf_{n} \| \phi_n Q(\widehat{p}_n)\| \geq \liminf_{n} \alpha \left(1-\frac{1}{2^{m_n-2}}\right) = \alpha.$$ Finally, $(\phi_j)$ being a subsequence of $(\varphi_{k})$ forces to
$$\alpha\leq \sup\left\{\lim_{n}\sup_{k\in \mathbb{N}} \|\phi_k Q(p_n)\|: (p_n)\in \mathcal{D} \right\} \leq \alpha,$$ contradicting the assumption made at the beginning of the proof.
\end{proof}

\begin{remark}
{ It should be notice here that in the statement of Lemma \ref{l 3.4}$(b)$ {\rm(}also in the corresponding statement of Lemma \ref{l 2.4}$(b)${\rm)}, we can not affirm, in general, that the whole sequence $\varphi_{k}$ satisfies the conclusion of the statement. A similar restriction applies to Lemma 2.4 in \cite{Randri}. So, a subtle difficulty appears at a certain stage of the proof of \cite[Theorem 3.1]{Randri}. Concretely, when in the just quoted paper, Proposition 3.4 is applied at the beginning of page 146 (even combined with Lemma 2.4), we can only guarantee that a subsequence of $(\varphi_n)$ fulfills the statement. Our proof extends the results to the setting of JBW$^*$-algebras and provides an argument to avoid the just quoted difficulties.}
\end{remark}


We isolate a property which was already implicitly stated in the results established above.

\begin{proposition}\label{p peirce 1 0 increase the measure of non-rwc} Suppose that $M$ is a JBW$^*$-algebra admitting a faithful normal functional. Let $(\varphi_k)$ be a bounded sequence of positive functionals in $M_*$ and let $(p_n)$ be a sequence in $\mathcal{D}$. Then for each $n_0$ in $\mathbb{N}$ the set $\displaystyle \left\{ \varphi_k \left(P_1 (p_{n_0}) + P_0 (p_{n_0})\right): k\in \mathbb{N} \right\}$ is not relatively weakly compact if, and only if, there exists a sequence $(q_n)$ in $\mathcal{D}$ satisfying $q_n\perp p_n$ for every $n\geq n_0$ and $\displaystyle  \lim_{ n\to\infty}\sup_{k} \|\varphi_k Q(q_n)\|>0.$
\end{proposition}

\begin{proof} Suppose there exists a sequence $(q_n)$ satisfying the conditions of the statement. Since $(q_n)\subseteq P_0 (p_{n_0}) (M),$ the sufficient condition is clear by Proposition \ref{p equiv rwc and uniform integrable}. To prove the reciprocal implication assume that the set  $\displaystyle \left\{ \varphi_k \left(P_1 (p_{n_0}) + P_0 (p_{n_0})\right): k\in \mathbb{N} \right\}$ is not relatively weakly compact. It is not restrictive to assume that $\|\varphi_k\| \leq 1$, for every $k$. By Proposition \ref{p equiv rwc and uniform integrable}, there exist $(e_n)$ in $\mathcal{D}$ such that $$\lim_{n\to \infty} \sup_{k} \left\| \varphi_k \left(P_1 (p_{n_0}) + P_0 (p_{n_0})\right) Q (e_n) \right\| >0.$$ It is easy to check that $\varphi_k \left(P_1 (p_{n_0}) + P_0 (p_{n_0})\right) = 2 \varphi_k Q(p_{n_0},1-p_{n_0}) Q(1)+ \varphi_k Q(1-p_{n_0})Q(1)$, and hence, by Lemma \ref{l hidden 1}$(a)$ and $(c)$, $$\| \varphi_k \left(P_1 (p_{n_0}) + P_0 (p_{n_0})\right) Q(e_n) \| \leq  2 \|\varphi_k Q(p_{n_0},1-p_{n_0}) Q(e_n)\|$$ $$+ \| \varphi_k Q(1-p_{n_0}) Q(e_n)\| \leq 6 \|\varphi_k Q(Q(1-p_{n_0})(e_n))\|^{\frac14} + \| \varphi_k Q(Q(1-p_{n_0})(e_n))\|^{\frac12} $$ $$ \leq 7 \|\varphi_k Q(Q(1-p_{n_0})(e_n))\|^{\frac14}.$$ Therefore, denoting $a_n := Q(1-p_{n_0})(e_n)$, it follows that $(a_n)$ is a strong$^*$-null decreasing sequence with $0\leq a_n\leq 1-p_{n_0}$ and $\displaystyle \liminf_{n} \sup_{k} \|\varphi_k Q(a_n)\| >0.$ Applying Lemma \ref{l 3.4} we find a sequence $(q_n)$ in $\mathcal{D}$ such that $q_n \leq r(a_1) \leq 1-p_{n_0}$ {\rm($\forall n$)} and $\displaystyle \lim_{n\to \infty} \sup_{k} \| \varphi_{k} Q(q_n) \| >0$.
\end{proof}

Our next result is a first Kadec-Pelczy{\'n}ski dichotomy-type theorem type for bounded sequences of positive functionals in the predual of a JBW$^*$-algebra.

\begin{theorem}\label{c KP dichotomy for positive sequences} Let $(\varphi_n)$ be a bounded sequence of positive functionals in the predual of a JBW$^*$-algebra $M$. Then there exist a subsequence $(\varphi_{n_k})$ and a decreasing strong$^*$-null sequence of projections $(q_n)$ in $M$ such that \begin{enumerate}[$(a)$]
\item the set $\{\varphi_{n_k} - \varphi_{n_k} P_{2} (q_k): k\in \mathbb{N}\}$ is relatively weakly compact,
\item $\varphi_{n_k}=\phi_k+\psi_k$, with $\phi_k := \varphi_{n_k} - \varphi_{n_k} P_{2} (q_k)$, $\psi_k := \varphi_{n_k} P_{2} (q_k),$ $\phi_k Q(q_k)= 0$ and $\psi_k Q(q_k)^2 = \psi_k$, for every $k$.
\end{enumerate}
\end{theorem}

\begin{proof} We may assume, without loss of generality, that $\|\varphi_n\|\leq 1$ and the set $\{\varphi_n: n\in \mathbb{N}\}$ is not relatively weakly compact. By Theorem \ref{t 3.1} (see also Proposition \ref{p equiv rwc and uniform integrable}) there exists a projection $q_0$ in $M$ such that $M_0= Q(q_0) (M)$ is a JBW$^*$-algebra which admits a faithful normal state, $(\varphi_k)$ is supported by $M_0$ with the following property: there exists a sequence $(q_n)\subseteq M_0$ in $\mathcal{D}$ and a subsequence $(\varphi_{\tau(k)})$ satisfying \begin{equation}\label{eq 1 first KP dichotomy}\sup\left\{\lim_{n}\sup_{k\in \mathbb{N}} \|\varphi_{\tau(k)} Q(r_n)\|: (r_n)\in \mathcal{D} \right\} = \lim_{n}\sup_{k\in \mathbb{N}} \|\varphi_{\tau(k)} Q(q_n)\|
\end{equation}
$$= \sup\left\{\lim_{n}\sup_{k\in \mathbb{N}} \|\varphi_{\tau(k)} Q(r_n)\|: (r_n)\in \mathcal{D}, (r_n)\subseteq M_0 \right\}= \alpha>0.$$

Set $\phi_k := \varphi_{\tau(k)} - \varphi_{\tau(k)} P_{2} (q_k)$ and $\psi_k := \varphi_{\tau(k)} P_{2} (q_k).$ Fix $k_0$ in $\mathbb{N}$. Clearly, for each $k\geq k_0$, $\phi_k \left(P_{0} (q_{k_0})+P_{1} (q_{k_0})\right) = \varphi_{\tau(k)} \left(P_{0} (q_{k_0})+P_{1} (q_{k_0})\right).$ By Proposition \ref{p peirce 1 0 increase the measure of non-rwc}, the set $$\left\{ \phi_{k} \left(P_1 (q_{k_0}) + P_0 (q_{k_0})\right): \begin{array}{c}
                                                                           k\in \mathbb{N} \\
                                                                           k\geq k_0
                                                                         \end{array} \right\} = \left\{ \varphi_{\tau(k)} \left(P_1 (q_{k_0}) + P_0 (q_{k_0})\right): \begin{array}{c}
                                                                           k\in \mathbb{N} \\
                                                                           k\geq k_0
                                                                         \end{array}
 \right\}$$ is not relatively weakly compact if, and only if, there exists a sequence $(\widetilde{q}_n)\subseteq M_0$ in $\mathcal{D}$ satisfying $\widetilde{q}_k\perp q_k$ for every $k\geq k_0$ and $$  \lim_{ n\to\infty}\sup_{k} \|\varphi_{\tau(k)} Q(\widetilde{q}_n)\|= \alpha_1 >0.$$ The sequence $(\widetilde{r}_n)_{n\geq k_0}= (q_n+\widetilde{q}_n)_{n\geq k_0}\subseteq M_0$ lies in $\mathcal{D}$ and $$\lim_{k_0\leq n\to \infty}\sup_{k\in \mathbb{N}} \|\varphi_{\tau(k)} Q(\widetilde{r}_n)\| \geq \lim_{k_0\leq n\to \infty}\sup_{k\in \mathbb{N}} \|\varphi_{\tau(k)} Q({q}_n)\|$$ $$ + \lim_{k_0\leq n\to \infty}\sup_{k\in \mathbb{N}}  \|\varphi_{\tau(k)} Q(\widetilde{q}_n)\|\geq \alpha + \alpha_1 > \alpha,$$ contradicting $(\ref{eq 1 first KP dichotomy})$. Therefore the set $\left\{ \phi_{k} \left(P_1 (q_{k_0}) + P_0 (q_{k_0})\right): k\in \mathbb{N} \right\}$ is relatively weakly compact. It follows, by Lemma \ref{l 2.3}, that the set  $\{\phi_k: k\in \mathbb{N}\}$ is relatively weakly compact if, and only if, $\displaystyle\exists\lim_{n\to\infty} \sup_{k} \|\phi_{k} Q(q_n)\|=0.$ In order to finish the proof, we shall show that the latter limit exists and is zero.\smallskip

Suppose, on the contrary, that there exists $\theta>0$ and a subsequence $(q_{\sigma(n)})$ satisfying $\displaystyle\sup_{k} \|\phi_{k} Q(q_{\sigma(n)})\| \geq \theta,$ and $\sigma(n)>\tau(n),$ for every positive integer $n.$ Since $\lim_{n}\sup_{k} \|\varphi_{\tau(k)} Q(q_{\sigma (n)})\| = \alpha >0,$ by Lemma \ref{l 2.4}$(b)$, we can find subsequences $(\varphi_{\eta\tau(k)})$ and $(q_{\kappa\sigma(n)})$ such that $\exists \lim_{n} \|\varphi_{\eta\tau(n)} Q(q_{\kappa\sigma(n)})\| = \alpha$. To simplify notation, we assume that $(\varphi_{\eta\tau(k)})=(\varphi_{\tau(k)})$ and $(q_{\kappa\sigma(n)})=(q_{\sigma(n)})$, and hence \begin{equation}\label{eq 3 corollary poositive}\exists \lim_{n} \|\varphi_{\tau(n)} Q(q_{\sigma(n)})\| = \alpha.
\end{equation}

We observe that, for each $n\geq k$, $P_2 (q_k) Q(q_n)= Q(q_n)$, thus $\phi_k Q(q_n)= \varphi_{\tau(k)} Q(q_n) - \varphi_{\tau(k)} P_2 (q_k) Q(q_n) =0, $ which shows that $$\sup_{k} \|\phi_{k} Q(q_{\sigma(n)})\| = \sup_{k>{\sigma(n)}} \|\phi_{k} Q(q_{\sigma(n)})\|.$$ We can therefore find a subsequence $(\phi_{\sigma_1(k)})$ such that $\sigma_1(n)>\sigma(n),$ and $\|\phi_{\sigma_1(n)} Q(q_{\sigma(n)})\|\geq \theta>0,$ for every $n\in \mathbb{N}.$ Writing $$\phi_{\sigma_1(n)} Q(q_{\sigma(n)}) =  \varphi_{\tau\sigma_1(n)} Q(q_{\sigma(n)})- \varphi_{\tau\sigma_1(n)}  Q(q_{\tau\sigma_1(n)})^2 Q(q_{\sigma(n)}) $$ $$= \varphi_{\tau\sigma_1(n)} Q(q_{\sigma(n)})- \varphi_{\tau\sigma_1(n)} Q(q_{\tau\sigma_1(n)})$$ $$= \varphi_{\tau\sigma_1(n)} Q(q_{\sigma(n)}-q_{\tau\sigma_1(n)}) + 2 \varphi_{\tau\sigma_1(n)} Q(q_{\sigma(n)}-q_{\tau\sigma_1(n)},q_{\tau\sigma_1(n)}). $$

We deduce from Lemma \ref{l hidden 1}$(a)$ and $(c)$ that \begin{equation}\label{eq 2 corollary poositive}\theta \leq \|\phi_{\sigma_1(n)} Q(q_{\sigma(n)})\|\leq  \left\| \varphi_{\tau\sigma_1(n)} Q(q_{\sigma(n)}-q_{\tau\sigma_1(n)})\right\|
 \end{equation} $$+2 \left\| \varphi_{\tau\sigma_1(n)} Q(q_{\sigma(n)}-q_{\tau\sigma_1(n)},q_{\tau\sigma_1(n)})\right\|$$ $$\leq 7 \left\| \varphi_{\tau\sigma_1(n)} Q(q_{\sigma(n)}-q_{\tau\sigma_1(n)})\right\|^{\frac14}.$$ Having in mind that $q_{\sigma\sigma_1(n)}\leq q_{\sigma_1(n)}\leq q_{\sigma(n)},$ we have $$\alpha \geq \liminf_{n} \|\varphi_{\tau(\sigma_1(n))} Q(q_{\sigma (n)})\| $$ $$\geq \liminf_{n} \|\varphi_{\tau\sigma_1(n)} Q(q_{\tau\sigma_1 (n)})\| + \|\varphi_{\tau\sigma_1(n)} Q(q_{\sigma(n)}-q_{\tau\sigma_1 (n)})\|$$ $$\hbox{(by  $(\ref{eq 2 corollary poositive})$)} \geq \liminf_{n} \|\varphi_{\tau\sigma_1(n)} Q(q_{\tau\sigma_1 (n)})\| + \frac{\theta^4}{7^4} $$  $$\geq \liminf_{n} \|\varphi_{\sigma\sigma_1(n)} Q(q_{\tau\sigma_1 (n)})\| + \frac{\theta^4}{7^4} =\hbox{(by  $(\ref{eq 3 corollary poositive})$)}= \alpha + \frac{\theta^4}{7^4},$$ which is impossible.
\end{proof}

\section{A dichotomy-type theorem for general bounded sequences}

We recall that a functional $\omega$ in the predual of a JBW$^*$-algebra $M$ is said to be symmetric or hermitian
when $\omega (x^*) = \overline{\omega(x)}$, for every $x\in M$. Every hermitian functional $\omega$ in $M_*$ has a unique decomposition
(the Jordan decomposition) into $\omega = \omega^{+}- \omega^{-}$, where $\omega^{+}$ and
$\omega^{-}$ are two positive functionals in $M_*$ with $\|\omega\| = \| \omega^{+}\|  + \|\omega^{-}\|$ (compare,
for example, \cite[Proposition 4.5.3]{HancheStor}). Having in mind that each functional $\phi$ in $M_*$ has a unique decomposition
into $\omega_1 + i\omega_2$, where $\omega_1$ and $\omega_2$ are hermitian functionals, $\phi$ admits a canonical decomposition
$\phi=\phi^{1} - \phi^{2} + i(\phi^{3}- \phi^{4}),$ where each $\phi^{j}$ is
positive. The symbol $|\phi|$ will stand for the positive functional $\phi^{1} + \phi^{2} + \phi^{3} + \phi^{4},$ and we call $|\phi|$ the
absolute value of $\phi$.\smallskip

We shall require the following consequence of Lemma \ref{l hidden 1}, whose proof is left to the reader.

\begin{lemma}\label{l hidden 1 b} Let $p$ be a projection in a JBW$^*$-algebra $M$, and let $\phi$ be a (normal) functional on $M$ with $\|\phi\|\leq 1$.
Then the following inequalities hold for each $a$ in $M$ with $0\leq a\leq 1$: \begin{enumerate}[$(a)$]
\item $\displaystyle\|\phi Q(p) Q(a) \| \leq 4 \ \||\phi| Q(p) Q(a) Q(p)\|^{\frac12};$
\item $\displaystyle\left\|\phi Q(1,p) Q(a) \right\|^{2}\leq 64 \ \left\| |\phi| Q(p) Q(a) \right\|;$
\item $\displaystyle\|\phi Q(1-p,p) Q(a) \| \leq {8}\ \||\phi| Q(p) Q(a) Q(p)\|^{\frac14}+ 4 \ \||\phi| Q(p) Q(a) Q(p)\|^{\frac12}.\hfill\Box$
\end{enumerate}
\end{lemma}

A Kadec-Pelczy{\'n}ski dichotomy-type theorem for bounded sequences in the predual of a JBW$^*$-algebra can be stated now.

\begin{theorem}\label{t KP dichotomy for general sequences 1} Let $(\phi_n)$ be a bounded sequence of functionals in the predual of a JBW$^*$-algebra $M$. Then there exist a subsequence $(\phi_{n_k})$, bounded sequences $(\xi_k)$ and $(\psi_k)$ in $M_*$ and a decreasing strong$^*$-null sequence of projections $(q_n)$ in $M$ such that $\phi_{n_k}=\xi_k+\psi_k$, $\xi_k Q(q_k)= 0$, $\psi_k Q(q_k)^2 = \psi_k$ and the set $\{\xi_k: k\in \mathbb{N}\}$ is relatively weakly compact.
\end{theorem}

\begin{proof} We may assume that $\|\phi_n\|\leq 1$, for every $n\in \mathbb{N}$.
Applying Theorem \ref{c KP dichotomy for positive sequences} to the sequence $(|\phi_n|)$, we find a subsequence $(|\phi_{n_k}|)$ and a decreasing, strong$^*$-null sequence of projections $(q_k)$ in $M$ such that the set $\{|\phi_{n_k}| - |\phi_{n_k}| P_{2} (q_k): k\in \mathbb{N}\}$ is relatively weakly compact. Set $\xi_{k} := \phi_{n_k} - \phi_{n_k} P_{2} (q_k)$ and $\psi_k := \phi_{n_k} P_{2} (q_k)$\smallskip

We claim that, for each $k_0$, the set $\{\xi_{k} (P_{0} (q_{k_0})+ P_{1} (q_{k_0})) : k\geq k_0\}$ is relatively weakly compact. Otherwise, by Proposition \ref{p equiv rwc and uniform integrable}, there exists a subsequence $(\xi_{\sigma(k)})$ with $\sigma(1) \geq k_0$, and a decreasing, strong$^*$-null sequence of projections $(p_n)$ in $M$ such that $$0<\gamma \leq \left\| \xi_{\sigma(k)} (P_{0} (q_{k_0})+ P_{1} (q_{k_0})) Q(p_k) \right\|$$
$$=\hbox{(since $\sigma(k) \geq k_0$)}= \left\| \phi_{\sigma(n_k)} (P_{0} (q_{k_0})+ P_{1} (q_{k_0})) Q(p_k) \right\|$$
$$= \left\| \phi_{\sigma(n_k)} (Q (1-q_{k_0})+ 2 Q (1-q_{k_0}, q_{k_0})) Q(p_k) Q(1)\right\|$$
$$ \leq \left\| \phi_{\sigma(n_k)} Q (1-q_{k_0})  Q(p_k) \right\| + 2 \left\| \phi_{\sigma(n_k)}Q (1-q_{k_0}, q_{k_0})) Q(p_k)\right\|$$
$$\hbox{(by Lemma \ref{l hidden 1 b})} \leq 12 \left\| |\phi_{\sigma(n_k)}| Q ( Q(1-q_{k_0})  p_k) \right\|^{\frac12} $$
$$+ 16 \left\| |\phi_{\sigma(n_k)}| Q ( Q(1-q_{k_0})  p_k) \right\|^{\frac14} \leq 28 \left\| |\phi_{\sigma(n_k)}| Q ( Q(1-q_{k_0})  p_k) \right\|^{\frac14},$$
which contradicts the fact that $\{|\phi_{n_k}| - |\phi_{n_k}| P_{2} (q_k): k\in \mathbb{N}\}$ is relatively weakly compact, because $( Q(1-q_{k_0})  p_k)$ is a strong$^*$-null sequence in $M_0(q_{k_0})$ $\subseteq M_0(q_{k}), $ for every $k\geq k_0$.\smallskip

The statement of the theorem will follow from Lemma \ref{l 2.3} once the condition $\displaystyle \lim_{n} \sup_{k} \|\xi_k Q(q_n)\| = 0$ is fulfilled. Arguing by contradiction, suppose that there exist $\theta >0$
and a subsequence $(q_{\sigma(n)})$ such that $$ \sup_{k} \|\xi_k Q(q_{\sigma(n)})\| \geq \theta, \hbox{  for every $n$.}$$

Since for each $n\geq k$, $\xi_{k} Q(q_n)= \varphi_{n_k} Q(q_n) - \varphi_{n_k} P_2 (q_k) Q(q_n) =\varphi_{n_k} Q(q_n) - \varphi_{n_k} Q(q_n) =0,$ (and hence $\displaystyle \sup_{k} \|\xi_{k} Q(q_{\sigma(n)})\| = \sup_{k>{\sigma(n)}} \|\xi_{k} Q(q_{\sigma(n)})\|$), we can find a subsequence $(\xi_{\sigma_1(k)})$ satisfying
$\|\xi_{\sigma_1(k)} Q(q_{\sigma(k)})\| \geq \theta$ and $\sigma_1(k) > \sigma(k)$, for every $k$. We therefore have: $$0<\theta \leq \|\xi_{\sigma_1(k)} Q(q_{\sigma(k)})\| = \big\| \phi_{n_{\sigma_1(k)}} Q(q_{\sigma(k)})- \phi_{n_{\sigma_1(k)}} P_{2} (q_{\sigma_1(k)}) Q(q_{\sigma(k)})\big\|$$ $$= \big\| \phi_{n_{\sigma_1(k)}} \left(Q(q_{\sigma(k)})-Q (q_{\sigma_1(k)})\right)\big\| \leq  \big\| \phi_{n_{\sigma_1(k)}} Q(q_{\sigma(k)}-q_{\sigma_1(k)})\big\| $$ $$+ 2 \big\| \phi_{n_{\sigma_1(k)}}  Q (q_{\sigma(k)}-q_{\sigma_1(k)},q_{\sigma_1(k)})\big\| \leq \hbox{(by Lemma \ref{l hidden 1 b}, with $a=1$)}$$
$$ \leq 12 \ \big\| |\phi_{n_{\sigma_1(k)}}| \ Q(q_{\sigma(k)}-q_{\sigma_1(k)})^2\big\|^{\frac12} + 16 \ \big\| |\phi_{n_{\sigma_1(k)}}| \ Q(q_{\sigma(k)}-q_{\sigma_1(k)})^2\big\|^{\frac14} $$ $$ \leq 28 \ \big\| |\phi_{n_{\sigma_1(k)}}| \ Q(q_{\sigma(k)}-q_{\sigma_1(k)})^2\big\|^{\frac14}.$$ Noticing that $P_2 (q_{\sigma_1(k)}) Q(q_{\sigma(k)}-q_{\sigma_1(k)}) =0$ (recall that $q_{\sigma(k)}-q_{\sigma_1(k)}\perp q_{\sigma_1(k)}$), we deduce that $$\big\| \left(|\phi_{n_{\sigma_1(k)}}| - |\phi_{n_{\sigma_1(k)}}| P_2 (q_{\sigma_1(k)})\right) Q(q_{\sigma(k)}-q_{\sigma_1(k)})\big\| $$ $$= \big\| |\phi_{n_{\sigma_1(k)}}| \ Q(q_{\sigma(k)}-q_{\sigma_1(k)})\big\| \geq \frac{\theta^{4}}{28^4}>0,$$ for every $k\in \mathbb{N},$ contradicting again the fact that $\{|\phi_{n_k}| - |\phi_{n_k}| P_{2} (q_k): k\in \mathbb{N}\}$ is relatively weakly compact.
\end{proof}

A standard argument allows us to deduce the following Kadec-Pelczy{\'n}ski dichotomy-type theorem with a sequence of mutually orthogonal projections in its thesis (cf. \cite[Theorem 3.9]{RandriJOT}), an sketch of the proof is rather included here for completeness reasons.

\begin{corollary}\label{c KP dichotomy orthogonal for general sequences} Let $(\phi_n)$ be a bounded sequence of functionals in the predual of a JBW$^*$-algebra $M$. Then there exist a subsequence $(\phi_{\tau(n)})$,
and a sequence of mutually orthogonal projections $(p_n)$ in $M$ such that: \begin{enumerate}[$(a)$]
\item the set $\left\{\phi_{\tau(n)} - \phi_{\tau(n)} P_{2} (p_n): n\in \mathbb{N}\right\}$ is relatively weakly compact,
\item $\phi_{\tau(n)}=\xi_n+\psi_n$, with $\xi_n := \phi_{\tau(n)} - \phi_{\tau(n)} P_{2} (p_n)$, and $\psi_n := \phi_{\tau(n)} P_{2} (p_n),$ {\rm(}$\xi_n Q(p_n)= 0$ and $\psi_n Q(p_n)^2 = \psi_n${\rm)}, for every $n$.
\end{enumerate}
\end{corollary}

\begin{proof} We may assume that $\|\phi_n\|\leq 1$, for every $n\in \mathbb{N}$. Find, by Theorem \ref{t KP dichotomy for general sequences 1}, a subsequence $(\phi_{\sigma(n)})$, bounded sequences $(\xi_n)$ and $(\psi_n)$ in $M_*$ and a decreasing strong$^*$-null sequence of projections $(q_n)$ in $M$ such that $\phi_{\sigma(n)}=\xi_n+\psi_n$, $\xi_n := \phi_{\sigma(n)} - \phi_{\sigma(n)} P_{2} (q_n)$, $\psi_n := \phi_{\sigma(n)} P_{2} (q_n),$ $\xi_n Q(q_n)= 0$, $\psi_n Q(q_n)^2 = \psi_n$ and the set $\{\xi_n: n\in \mathbb{N}\}$ is relatively weakly compact.\smallskip

Fix $n\in \mathbb{N}$. For each $m\geq n$, we have $$\lim_{m\to \infty} \| \phi_{\sigma(n)} P_{2} (q_n) - \phi_{\sigma(n)} P_{2} (q_n-q_m)\| = \lim_{m\to \infty}
\| \phi_{\sigma(n)} Q (q_n) - \phi_{\sigma(n)} Q (q_n-q_m)\|$$ $$\leq \lim_{m\to \infty} \| \phi_{\sigma(n)} Q (q_m) \| + 2 \| \phi_{\sigma(n)} Q (q_n-q_m, q_m)\| \leq \hbox{(by Lemma \ref{l hidden 1 b})}$$ $$\leq 28 \lim_{m\to \infty} \big\| |\phi_{\sigma(n)}| \ Q(q_{m})^2\big\|^{\frac14} = 28 \lim_{m\to \infty} \sqrt[4]{|\phi_{\sigma(n)}| (q_{m})}= 0. $$ We can therefore find a strictly increasing sequence $(\sigma_1(n))$ in $\mathbb{N}$ such that $$\lim_{n\to\infty} \big\| \phi_{\sigma\sigma_1(n)} Q (q_{\sigma_1(n)}) - \phi_{\sigma\sigma_1 (n)} Q (q_{\sigma_1(n)}-q_{\sigma_1(n+ 1)})\big\| =0.$$ Finally, taking $(p_n):= (q_{\sigma_1(n)}-q_{\sigma_1(n+ 1)})$, we get a sequence of mutually orthogonal projections in $M$. Since the set $\{\phi_{\sigma\sigma_1(n)} - \phi_{\sigma\sigma_1(n)} P_{2} (q_{\sigma_1(n)}): n\in \mathbb{N}\}$ is relatively weakly compact and $$\lim_{n\to\infty} \big\| \phi_{\sigma\sigma_1(n)} P_2 (q_{\sigma_1(n)}) - \phi_{\sigma\sigma_1 (n)} P_2 (p_{n})\big\| $$ $$= \lim_{n\to\infty} \big\| \phi_{\sigma\sigma_1(n)} Q (q_{\sigma_1(n)}) - \phi_{\sigma\sigma_1 (n)} Q (p_{n})\big\| =0,$$ the set $\{\phi_{\sigma\sigma_1(n)} - \phi_{\sigma\sigma_1(n)} P_{2} (p_{n}): n\in \mathbb{N}\}$ is relatively weakly compact. The proof concludes defining $(\phi_{\tau(n)}):= (\phi_{\sigma\sigma_1(n)}).$
\end{proof}

The lacking of a cone of positive in the more general setting of JBW$^*$-triple preduals makes unable the reasonings and proofs of this paper to establish a Kadec-Pelczy{\'n}ski dichotomy-type theorem for JBW$^*$-triple preduals. However, the following conjecture is natural to be posed:

\begin{conjecture} Let $(\phi_n)$ be a bounded sequence of functionals in the predual of a JBW$^*$-triple $W$. Then there exist a subsequence $(\phi_{\tau(n)})$, and a sequence of mutually orthogonal tripotents $(e_n)$ in $W$ such that: \begin{enumerate}[$(a)$]
\item the set $\left\{\phi_{\tau(n)} - \phi_{\tau(n)} P_{2} (e_n): n\in \mathbb{N}\right\}$ is relatively weakly compact,
\item $\phi_{\tau(n)}=\xi_n+\psi_n$, with $\xi_n := \phi_{\tau(n)} - \phi_{\tau(n)} P_{2} (e_n)$, and $\psi_n := \phi_{\tau(n)} P_{2} (e_n),$ for every $n$.
\end{enumerate}
\end{conjecture}

\bigskip\bigskip

\end{document}